\documentclass[12pt]{amsart}
\usepackage{amssymb}
\usepackage{amsfonts}
\usepackage{latexsym}
\usepackage{amscd}
\usepackage[mathscr]{euscript}
\usepackage{xy} \xyoption{all}

\newcommand{\Vq}{V_{[q]}}
\newcommand{\Vp}{V_{[p]}}
\newcommand{\Vd}{V_{[d^\infty]}}
\newcommand{\al}{\alpha}
\newcommand{\be}{\beta}
\newcommand{\N}{{\mathbb{N}}}

\newcommand{\Z}{{\mathbb{Z}}}

\newcommand{\uloopr}[1]{\ar@'{@+{[0,0]+(-4,5)}@+{[0,0]+(0,10)}@+{[0,0] +(4,5)}}^{#1}}
\newcommand{\uloopd}[1]{\ar@'{@+{[0,0]+(5,4)}@+{[0,0]+(10,0)}@+{[0,0]+ (5,-4)}}^{#1}}
\newcommand{\dloopr}[1]{\ar@'{@+{[0,0]+(-4,-5)}@+{[0,0]+(0,-10)}@+{[0, 0]+(4,-5)}}_{#1}}
\newcommand{\dloopd}[1]{\ar@'{@+{[0,0]+(-5,4)}@+{[0,0]+(-10,0)}@+{[0,0 ]+(-5,-4)}}_{#1}}

\newcommand{\luloop}[1]{\ar@'{@+{[0,0]+(-8,2)}@+{[0,0]+(-10,10)}@+{[0, 0]+(2,2)}}^{#1}}

\newtheorem{lemma}{Lemma}[section]
\newtheorem{corollary}[lemma]{Corollary}
\newtheorem{theorem}[lemma]{Theorem}
\newtheorem{proposition}[lemma]{Proposition}
\newtheorem{remark}[lemma]{Remark}
\newtheorem{definition}[lemma]{Definition}
\newtheorem{definitions}[lemma]{Definitions}
\newtheorem{example}[lemma]{Example}

\def\Ext{\operatorname{Ext}}
\DeclareMathOperator{\Rad}{Rad}

\def\M{\mathcal{M}}

\def\N{{\mathbb N}}

\def\Hom{\operatorname{Hom}}
\def\Ext{\operatorname{Ext}}

\def\dualita#1#2{\mathrel{
                 \mathop{\vcenter{
                 \offinterlineskip
                 \hbox to 0.6truecm{\rightarrowfill}
                 \hbox to 0.6truecm{\leftarrowfill}}}%
                 \limits_{#2}^{#1}}}

\DeclareMathOperator{\Ker}{Ker}

\begin{document}

\title[Extensions of simple modules  over Leavitt path algebras]{Extensions of simple modules \\ over Leavitt path algebras}
\author{Gene Abrams$^*$}
\address{Department of Mathematics, University of Colorado,
Colorado Springs, CO 80918 U.S.A.}
\email{abrams@math.uccs.edu}
\thanks{
.  \ \  $^*$corresponding author \ \ abrams@math.uccs.edu \ \ 7192553182 \\  The first  author is partially supported by a Simons Foundation Collaboration Grants for Mathematicians Award \#208941.     The second and third authors are supported by Progetto di Eccellenza Fondazione Cariparo ``Algebraic structures
and their applications: Abelian and derived categories, algebraic entropy and representation of algebras'' and Progetto di Ateneo ``Categorie Differenziali Graduate'' CPDA105885.    Part of this work was carried out during a visit of the first author to the Universit\`{a} degli Studi di Padova.    The first author is pleased to take this opportunity to again express his thanks to the host institution, and its faculty, for its warm hospitality and support by the Visiting Scientist 2012-13 grant of the Universit\`{a} degli Studi di Padova.}

\author{Francesca Mantese}
\address{Dipartimento di Informatica, Universit\`{a} degli Studi di Verona, I-37134 Verona, Italy}
\email{francesca.mantese@univr.it}

\author{Alberto Tonolo}
\address{Dipartimento Matematica, Universit\`{a} degli Studi di Padova, I-35121, Padova, Italy}
\email{tonolo@math.unipd.it}




\begin{abstract}

Let $E$ be a  directed graph, $K$ any field, and let $L_K(E)$ denote  the Leavitt path algebra of $E$ with coefficients in $K$.     For each rational infinite   path $c^\infty$ of $E$ we explicitly construct a projective resolution of  the corresponding  Chen simple left $L_K(E)$-module $V_{[c^\infty]}$.
 Further, when $E$ is row-finite, for each irrational infinite  path $p$ of $E$ we explicitly construct a projective resolution of  the corresponding  Chen simple left $L_K(E)$-module $V_{[p]}$.    For Chen simple modules $S,T$ we describe ${\rm Ext}_{L_K(E)}^1(S,T)$ by presenting an explicit  $K$-basis.   For any  graph $E$ containing at least one cycle, this description guarantees the existence of  indecomposable left $L_K(E)$-modules of any prescribed finite length.

\end{abstract}

\dedicatory{Dedicated to Alberto Facchini on the occasion of his sixtieth birthday.}
\maketitle


\section{Introduction}\label{IntroSection}

Given any directed graph $E$ and field $K$, one may construct the {\it Leavitt path algebra of E with coefficients in K} (denoted $L_K(E)$), as first described in \cite{AAP1} and \cite{AMP}.   Since their introduction, various structural properties of the algebras $L_K(E)$ have been discovered, with a significant number of the results in the subject taking on the following form:  $L_K(E)$ has some specified algebraic property if and only if $E$ has some specified graph-theoretic property.  (The structure of the field $K$ often plays no role in results of this type.)  A few (of many) examples of such results include a description of those Leavitt path algebras which are simple; purely infinite simple; finite dimensional; prime; primitive; exchange; etc. 

Although there are graphs for which the structure of corresponding Leavitt path algebra is relatively pedestrian (e.g., is a direct sum of matrix rings either over  $K$, or over the Laurent polynomial algebra $K[x,x^{-1}]$, or some combination thereof), the less-mundane examples of Leavitt path algebras exhibit somewhat exotic behavior.   For instance,  the prototypical Leavitt path algebra $A = L_K(R_n)$ ($n\geq 2$),  which arises from the graph $R_n$ having one vertex and $n$ loops,  has the property that $A \cong A^n$ as left (or right) $A$-modules.    Analogous  ``super decomposability" properties are also  found in other important classes of Leavitt path algebras.      These types of  structural properties lead to a dearth (if not outright absence) of indecomposable one-sided $L_K(E)$-ideals, which  subsequently makes the search for simple (and, more generally, indecomposable) modules over Leavitt path algebras somewhat of a challenge.

For a graph $E$, an {\it infinite path in }$E$ is a sequence of edges $e_1e_2e_3\cdots$, for which $s(e_{i+1}) = r(e_i)$ for all $i\in \N$.   
 In \cite{C}, Chen produces, for each infinite path $p$ in $E$, a simple left $L_K(E)$-module $V_{[p]}$.   Further, Chen describes, for each sink vertex $w$ of $E$, a simple left $L_K(E)$-module $\mathcal N_w$.    

In Section \ref{projresSection} we produce explicit projective resolutions  for  Chen simple modules.  As a result, we will see in Theorem \ref{projresofVcinftyfromL(E)v}  that $V_{[c^\infty]}$ is finitely presented for any closed path $c$. Further, in Theorem \ref{Vpinftynotfinitelypresented}  we give necessary and sufficient conditions on a row-finite graph $E$ which ensure that $V_{[p]}$ is not finitely presented for an irrational infinite path $p$.    In Section \ref{ExtSection} we describe the extension groups ${\rm Ext}^1(S,T)$ corresponding to any pair of Chen simple modules $S,T$.   Using some general results about uniserial modules over hereditary rings, we conclude by showing (Corollary \ref{cor:uniserialoflengthN})  how our description of ${\rm Ext}^1(S,S)$ guarantees the existence of indecomposable $L_K(E)$-modules of any prescribed finite length. 

We set some notation.  A (directed) graph $E = (E^0, E^1, s,r)$ consists of a {\it vertex set} $E^0$, an {\it edge set} $E^1$, and {\it source} and {\it range} functions $s, r: E^1 \rightarrow E^0$.  For $v\in E^0$, the set of edges $\{ e\in E^1 \ | \ s(e)=v\}$ is denoted $s^{-1}(v)$. $E$ is called {\it finite} in case both $E^0$ and $E^1$ are finite sets.   $E$ is called {\it row-finite} in case $s^{-1}(v)$ is finite for every $v\in E^0$.   
A {\it path} $\alpha$ in $E$ is a sequence $e_1 e_2 \cdots e_n$ of edges in $E$ for which $r(e_i) = s(e_{i+1})$ for all $1 \leq i \leq n-1$.  
We say that such $\alpha$ has {\it length} $n$, and we write $s(\alpha) = s(e_1)$ and $r(\alpha) = r(e_n)$.  We view each vertex $v \in E^0$ as a path of length $0$, and denote $v = s(v) = r(v)$.  We denote the set of paths in $E$ by ${\rm Path}(E)$. A  path $\sigma = e_1 e_2 \cdots e_n$ in $E$   is {\it closed}  in case $r(e_n) = s(e_1)$.     Following \cite{C} (but not standard in the literature), a closed path $\sigma$ is called {\it simple} in case $\sigma \neq \beta^m$ for any closed path $\beta$ and integer $m\geq 2.$       A {\it sink} in $E$ is a vertex $w \in E^0$ for which the set $s^{-1}(w)$ is empty, while an {\it infinite emitter} in $E$ is a vertex $u \in E^0$ for which the set $s^{-1}(u)$ is infinite.  

\smallskip

For any field $K$ and  graph $E$ the Leavitt path algebra $L_K(E)$ has been the focus of sustained investigation since 2004.  We give here a basic description of $L_K(E)$; for additional information, see  \cite{AAP1} or \cite{TheBook}.       Let $K$ be a field, and let $E = (E^0, E^1, s,r)$ be a directed  graph with vertex set $E^0$ and edge set $E^1$.   The {\em Leavitt path $K$-algebra} $L_K(E)$ {\em of $E$ with coefficients in $K$} is  the $K$-algebra generated by a set $\{v\mid v\in E^0\}$, together with a set of symbols $\{e,e^*\mid e\in E^1\}$, which satisfy the following relations:

(V)   \ \ \  \ $vu = \delta_{v,u}v$ for all $v,u\in E^0$, \  

  (E1) \ \ \ $s(e)e=er(e)=e$ for all $e\in E^1$,

(E2) \ \ \ $r(e)e^*=e^*s(e)=e^*$ for all $e\in E^1$,

 (CK1) \ $e^*e'=\delta _{e,e'}r(e)$ for all $e,e'\in E^1$, and

(CK2)Ê\ \ $v=\sum _{\{ e\in E^1\mid s(e)=v \}}ee^*$ for every   $v\in E^0$ for which $0 < |s^{-1}(v)| < \infty$.

An alternate description of $L_K(E)$ may be given as follows.  For any graph $E$ let $\widehat{E}$ denote the ``double graph" of $E$, gotten by adding to $E$ an edge $e^*$ in a reversed direction for each edge $e\in E^1$.   Then $L_K(E)$ is the usual path algebra $K\widehat{E}$, modulo the ideal generated by the relations (CK1) and (CK2).

\smallskip

It is easy to show that $L_K(E)$ is unital if and only if $|E^0|$ is finite; in this case, $1_{L_K(E)} = \sum_{v\in E^0}v$.    Every element of $L_K(E)$ may be written as $\sum_{i=1}^n k_i \alpha_i \beta_i^*$, where $k_i$ is a nonzero element of $K$, and each of the $\alpha_i$ and $\beta_i$ are paths in $E$.  If $\alpha \in {\rm Path}(E)$ then we may view $\alpha \in L_K(E)$, and will often refer to such $\alpha$ as a {\it real path} in $L_K(E)$; analogously, for $\beta = e_1 e_2 \cdots e_n \in {\rm Path}(E)$ we often refer to the element $\beta^* = e_n^* \cdots e_2^* e_1^*$ of $L_K(E)$ as a {\it ghost path} in $L_K(E)$.     The  map $KE \rightarrow L_K(E)$ given by the $K$-linear extension of $\alpha \mapsto \alpha$ (for $\alpha \in {\rm Path}(E)$) 
 is an injection of $K$-algebras by \cite[Corollary 1.5.12]{TheBook}.

The ideas presented in the following few paragraphs come from \cite{C};  however, some of the  notation we use here differs from that used in \cite{C}, in order to make our presentation more notationally consistent with the general body of literature regarding Leavitt path algebras.



  

Let $p$ be an {\it infinite path in} $E$; that is, $p$ is a sequence $ e_1e_2e_3\cdots$, where $e_i \in E^1$ for all $i\in \N$, and for which $s(e_{i+1}) = r(e_i)$ for all $i\in \N$.   We emphasize that while the phrase {\it infinite path in} $E$ might seem to suggest otherwise, an infinite path in $E$ is not an element of ${\rm Path}(E)$, nor may it be interpreted  as   an element of the path algebra $KE$ nor  of the Leavitt path algebra $L_K(E)$.  (Such a path is sometimes called a {\it left}-infinite path in the literature.)  We denote the set of infinite paths in $E$ by $E^\infty$.   

 For $p = e_1e_2e_3\cdots \in E^\infty$  and $n\in \N$ we denote by $\tau_{\leq n}(p)$, or often more efficiently by $p_n$, the (finite) path $e_1e_2\cdots e_n$, while we denote by $\tau_{>n}(p)$ the infinite path $e_{n+1}e_{n+2}\cdots$.   We note that $\tau_{\leq n}(p)$ is an element of ${\rm Path}(E)$ (and thus may be viewed as an element of $L_K(E)$), and that $p$ is the concatenation $p = \tau_{\leq n}(p) \cdot \tau_{>n}(p)$.  

Let $c$ be a closed path in $E$.  Then the path $c c c \cdots$ is an infinite path in $E$, which we denote by $c^\infty$.   We call an infinite path of the form $c^\infty$ a {\it cyclic infinite} path.   For $c$ a closed path in $E$ let $d$ be the simple closed path in $E$ for which $c = d^n$.   Then $c^\infty = d^\infty$ as elements of $E^\infty$.

If $p$ and $q$ are infinite paths in $E$, we say that $p$ and $q$ are {\it tail equivalent} (written $p \sim q$) in case there exist integers $m,n$ for which $\tau_{>m}(p) = \tau_{>n}(q)$; intuitively, $p \sim q$ in case $p$ and $q$ eventually become the same infinite path.   Clearly $\sim$ is an equivalence relation on $E^\infty$, and we let $[p]$ denote the $\sim$ equivalence class of the infinite path $p$.

The infinite path $p$ is called {\it rational} in 
case $p \sim c^\infty$ for some closed path $c$.   By a previous observation, we may assume without loss of generality that such $c$ is a simple closed path.      In particular, for  any  closed path $c$ we have that $c^\infty$ is rational.    If $p \in E^\infty$ is not rational we say $p$ is {\it irrational}.   

\begin{example}\label{R2Example}
{\rm 
Let $R_2$ denote the ``rose with two petals" graph
$$\xymatrix{\bullet^v \ar@(ul,dl)_e \ar@(ur,dr)^f} \ .$$
\noindent
Then $q = efeffefffeffffe\cdots$ is an irrational infinite path in $R_2^\infty$.  Indeed, it is easy to show that there are uncountably many distinct irrational infinite paths in $R_2^\infty$.   We note additionally that there are infinitely many simple closed paths in ${\rm Path}(R_2)$, for instance, any path of the form $ef^i$ for $i\in \Z^+$.  
}
\end{example}

Let $M$ be a left $L_K(E)$-module.  For each $m\in M$ we define the $L_K(E)$-homomorphism $\hat{\rho}_m: L_K(E) \rightarrow M$, given by $\hat{\rho}_m(r) = rm$.   The restriction of the right-multiplication map $\hat{\rho}_m$ may also be viewed as an $L_K(E)$-homomorphism from any left ideal $I$ of $L_K(E)$ into $M$.  When $I = L_K(E)v$ for some vertex $v$ of $E$, we will denote $\hat{\rho}_m$ simply by $\rho_m$. 

\medskip
 
Following \cite{C}, for any infinite path $p$ in $E$  we construct a simple left $L_K(E)$-module $V_{[p]}$, as follows.   

\begin{definition} \label{Chendef}
{\rm 
Let $p$ be an infinite path in the graph $E$, and let $K$ be any field.  Let $V_{[p]}$ denote the $K$-vector space having basis $[p]$, that is, having basis consisting of distinct elements of $E^\infty$ which are tail-equivalent to $p$.    For any $v\in E^0$, $e\in E^1$, and $q = f_1f_2f_3\cdots  \in [p]$, define 
$$
 v \cdot q =
 \begin{cases}
q  &\text{if }  v = s(f_1) \\
0 &\text{otherwise,} 
\end{cases}
    \ \ \ 
 e\cdot q=
  \begin{cases}
eq  &\text{if }  r(e) = s(f_1) \\
0 &\text{otherwise,} 
\end{cases}
\ \ \ \ \mbox{and} \qquad 
e^* \cdot q = 
 \begin{cases}
\tau_{>1}(q)  &\text{if }  e = f_1 \\
0 &\text{otherwise.} 
\end{cases}
$$
\noindent
Then  the $K$-linear extension to all of $V_{[p]}$ of this action endows $V_{[p]}$ with the structure of a left $L_K(E)$-module.   
}
\end{definition}

\begin{theorem} \label{Chentheoremforsimples}  (\cite[Theorem 3.3]{C}).  Let $E$ be any directed graph and $K$ any field.  Let $p\in E^\infty$.   Then the left $L_K(E)$-module $V_{[p]}$ described in  Definition \ref{Chendef}
 is simple.   Moreover, if $p,q \in E^\infty$, then $V_{[p]} \cong V_{[q]}$ as left $L_K(E)$-modules if and only if $p \sim q$, which happens precisely when $V_{[p]} = V_{[q]}$.
\end{theorem}

We will refer to a module of the form $V_{[p]}$ as presented in Theorem \ref{Chentheoremforsimples} as a {\it Chen simple module}.   

For any sink $w$ in a graph $E$,  Chen in \cite{C} presents a construction, similar in flavor to the one given in Definition \ref{Chendef},  of a simple left $L_K(E)$-module $\mathcal{N}_w$.   He then shows that $\mathcal{N}_w$ is isomorphic as a left $L_K(E)$-module to the left ideal $L_K(E)w$ of $L_K(E)$ generated by $w$.  Observe that,  for any sink $w$, the ideal $L_K(E)w$ is spanned by the paths in $E$ ending in $w$. Moreover for any $i\in \mathbb{Z}^+$, we get that $w^i=w$ and thus we can consider $w=w^\infty$ as an element in $E^{\infty}$. For these reasons, for any sink $w$ of $E$, we refer to $\mathcal{N}_w=L_K(E)w$ as a Chen simple module  and, for consistency, we  denote $\mathcal{N}_w$ by $V_{[w^\infty]}$.

 \begin{remark}\label{rem:Bergman}
 {\rm  By invoking a powerful result of Bergman, it was established in \cite[Theorem 3.5]{AMP} that, when $E$ is row-finite, then $L_K(E)$ is hereditary, i.e., every left ideal of $L_K(E)$ is projective.  This presumably could make the search for projective resolutions of various $L_K(E)$-modules somewhat easier, in that the projectivity of left ideals is already a given.  However, much of the strength of our results lies in our explicit description of the kernels of germane maps; for instance, it is these explicit descriptions which will allow us to analyze the ${\rm Ext}^1$ groups of the Chen simple modules.   
  } 
 \end{remark}

A significant majority of the structural properties of a Leavitt path algebras  $L_K(E)$ do not rely on the specific structure of the field $K$.   The results contained in this article are no exceptions.     So while each of the statements of the results made herein should also contain the explicit hypothesis ``Let $K$ be any field", we suppress this statement throughout for efficiency.    For a field $K$, $K^\times$ denotes the nonzero elements of $K$.   

%
%
%
%

\section{Projective resolutions of Chen simple modules over $L_K(E)$}\label{projresSection}

The goal of this section is to present an explicit description  of a projective resolution of $S$, where $S$ is a Chen simple module over the Leavitt path algebra $L_K(E)$.   Such an explicit description will provide a strengthening of some previously established results (see  \cite[Proposition 4.1]{AR}), as well as provide the necessary foundation for subsequent results. 
As we shall see, the description of projective resolutions, as well as the  description of the ${\rm Ext}^1$ groups, of Chen simple modules will proceed based on which of the following three types describes the module:  
\begin{enumerate}
\item $V_{[w^\infty]}\cong L_K(E) w$ where $w$ is a sink,
\item $V_{[c^\infty]}$ where $c$ is a simple closed path;
\item  $V_{[q]}$ where $q$ is an irrational infinite path.  
\end{enumerate}

Let $v$ be any vertex in $E$.  Since $v$ is an idempotent, the left ideal $L_K(E)v$ is a projective left $L_K(E)$-module. Therefore projective resolutions of Chen simple modules of type 1 are easy:
\begin{proposition}[Type (1)]\label{projressink}
Let $w$ be a sink in $E$.  Then the Chen simple left module $V_{[w^\infty]}$ is projective. 
\end{proposition}
\begin{proof}
We have $V_{[w^\infty]} \cong L_K(E)w$ as left $L_K(E)$-modules by \cite{C}.
\end{proof}

We now begin the process of describing projective resolutions of Chen simple modules of the second type, namely, of the form $V_{[c^\infty]}$ for $c$ a simple closed path.

\medskip

{\textbf{Notations:}}  Let $c=e_1e_2 \cdots e_t$ be a simple closed path in $E$, with $v=s(e_1) = r(e_t)$.

\begin{enumerate}
\item For $0 \leq i \leq t$   we define $c_i : = e_1 e_2 \cdots e_i$ and $d_i := e_{i+1} e_{i+2} \cdots e_t$ (where $c_0 = v = d_t$ and $c_t = c = d_0$).   Then clearly $c = c_i d_i$ for each $0\leq i \leq t$.  

\item For $n \geq 0$ we let $c^{-n}$ denote $(c^*)^n$, and let $c^0$ denote $v = s(c)$.

\item  An element $\mu$ of $L_K(E)$  is  said to be a {\it standard form monomial} in case there exist  $\alpha, \beta \in {\rm Path}(E)$ for which  $\mu = \alpha \beta^*$.    We denote the set of standard form monomials in $L_K(E)$ by $\mathcal{S}$.  For $\mu = \alpha \beta^*$ a standard form monomial we define $r(\mu):= r(\beta^*) = s(\beta)$; that is, $r(\mu)$ is the unique element $v\in E^0$ for which $\mu v = \mu$.  
Define 
$$\mathcal{S}_1(c) := \{\mu \in \mathcal{S} \ | \ \mu \cdot c^N = 0 \ \mbox{in} \  L_K(E) \ \mbox{for some} \ N \in \mathbb{N}\},  \ \mbox{and} \  \mathcal{S}_2(c) :=  \mathcal{S} \setminus \mathcal{S}_1(c).  $$
  Although   $\mathcal{S}_1(c)$ and $\mathcal{S}_2(c)$ depend on $c$, we will often simply write $\mathcal{S}_1$ and $\mathcal{S}_2$ for these sets. 
\end{enumerate}

\smallskip

By analyzing the form of monomials in $L_K(E)$, we get the following description of the elements of $\mathcal{S}_1$ and $\mathcal{S}_2$.    

\begin{lemma}\label{S_1andS_2description}  Let $ 0 \neq \mu \in \mathcal{S}$. If $c$ is a sink then $\mu \in \mathcal{S}_1$ if and only if $r(\mu) \neq c$.  If $c$ is a closed path $e_1e_2\cdots e_t$,  then
$\mu \in \mathcal{S}_1$ if and only if $\mu$ is of one of the following two forms:

\begin{enumerate}
\item $r(\mu) \neq s(c)$ (i.e.,  $\mu \cdot s(c) = 0$ in $L_K(E)$), or 
\item $\mu = \mu' f^* c_i^* (c^*)^n$ for some $n,i \in \Z^+$, some  $\mu' \in \mathcal{S}$,  and some $f \in E^1$ for which $s(f) =  s(e_{i+1})$ but $f\neq e_{i+1}$. 
\end{enumerate}
\noindent
Consequently, $0 \neq \mu \in \mathcal{S}_2$ if and only if $\mu = \alpha c_i^* (c^*)^n$ for some path $\alpha$ in $E$, and some pair of non-negative integers $n,i$.  
\end{lemma}

\begin{lemma}\label{M_1inR(c-1)Lemma}  Let $c$ be a closed path in the  graph $E$, and let $v = s(c)$.  

(1)   For any $z\in \mathbb{Z}$ we have $c^z - v \in L_K(E)(c-v)$.

(2)  Suppose $\mu \in \mathcal{S}_1(c)  $.  Then $\mu \in (\sum_{u\in E^0 \setminus \{v\}} L_K(E)u) \bigcup L_K(E)(c-v)$.   
\end{lemma}

\begin{proof}  (1)    If $z=0$ we have $s(c) - v = 0 = 0(c-v).$   For $z>0$ we have $c^z - v = (c^{z-1} + c^{z-2} + \hdots + c + v)(c-v)$.   For $z<0$ we have $c^z - v = - c^z(c^{-z} - v) $, which is in $L_K(E)(c-v)$ by the previous case.  

(2)  Suppose  $\mu \in \mathcal{S}_1(c)$.  If $r(\mu) \neq v$ then $\mu \in \sum_{u\in E^0 \setminus \{v\}} L_K(E)u$.   On the other hand, suppose $r(\mu) = v$, and that   $\mu \cdot c^N = 0$ for some $N \in \mathbb{N}$.  But $r(\mu)=v$ gives $\mu v = \mu$, so that  with the hypothesis  $\mu = - \mu(c^N - v)$, 
which gives that $\mu \in L_K(E)(c^N - v) \subseteq L_K(E)(c-v)$ by the previous paragraph.  
\end{proof}

\begin{remark} \label{pathsremark}
{\rm Let $p=e_1e_2\cdots$ be an infinite path in $E$. If $p=\tau_{>r}(p)$ for some $r>0$, then $p$ is a rational path of the form $p=c^{\infty}$, where $c$ is the closed path
$e_1e_2\cdots e_r$.  This follows from the observation that $p=\tau_{>r}(p)$ implies $p=\tau_{>ir}(p)$ for all $i\in \N$. }
\end{remark}

\begin{lemma}\label{pcinftyequalsqcinftyLemma}
Let  $c$ be a simple closed path $e_1 e_2 \cdots e_t$ in $E$ with $s(c)=r(c)=v$.  
Suppose $\alpha$ and $\beta$ are  paths in $E$ for which $0 \neq \alpha \cdot c^\infty = \beta \cdot c^\infty$ in $V_{[c^\infty]}$. 
Then there exists $N\in \mathbb{Z}^+$ for which $\alpha = \beta c^N$ or $\beta = \alpha c^N$.   

Consequently, $\alpha \cdot c^\infty = \beta \cdot c^\infty$  in $V_{[c^\infty]}$  implies 
$\alpha - \beta \in L_K(E)(c-v)$.
\end{lemma}
\begin{proof} Assume $\alpha=f_1f_2\cdots f_\ell$ and $\beta=g_1g_2\cdots g_m$,  where the  $f_i$ and  $g_j$ are edges in $E$. If $\ell=m$, from $\alpha \cdot c^\infty = \beta \cdot c^\infty$ we get $\alpha = \beta$. So assume $m>\ell$; we have 
\[g_1\cdots g_\ell g_{\ell+1}\cdots g_m=f_1\cdots f_\ell c^n {e_1}{e_2}\cdots {e_k}\]
with $m-\ell=k+n\times t$, $k\leq t$. If $k<t$, from $\alpha \cdot c^\infty = \beta \cdot c^\infty$ we get
\[c^\infty=c_{k+1}\cdots c_t\cdot c^\infty=\tau_{>k} (c^{\infty});\]
then by Remark \ref{pathsremark} we would  have $c^\infty=(e_1\cdots e_k)^\infty$, a contradiction since $c$ is simple. Therefore $k=t$ and $\beta=\alpha c^{n+1}$. The case $m< \ell$ is identical.


For the second statement, note that $0 \neq \alpha \cdot c^\infty = \beta \cdot c^\infty$ gives $r(\alpha) = r(\beta)$; denote this common vertex by $v$.   So  if $\alpha = \beta c^n$ then $\alpha - \beta = \beta (c^n - v)$, which is in $ L_K(E)(c-v)$ by Lemma \ref{M_1inR(c-1)Lemma}(1).  The case $\beta = \alpha c^n$ is identical.  
   \end{proof}
   
\begin{proposition}    \label{Vsubcinftypropforgeneralgraphs}  Let $E$ be any  graph.
Let $c$ be a simple closed path in  $E$, and let $v$ denote  $s(c) = r(c)$. 
Let $\rho_{c^\infty}: L_K(E)v \rightarrow V_{[c^\infty]}$ and
$\hat{\rho}_{c^\infty}: L_K(E) \rightarrow V_{[c^\infty]}$ denote the right multiplication by $c^\infty$.
Then 
$${\rm Ker}(\rho_{c^\infty}) =   L_K(E)(c-v)\quad\text{and}\quad
 {\rm Ker}(\hat{\rho}_{c^\infty}) = \left(\sum_{u\in E^0 \setminus \{v\}} L_K(E)u \right) \oplus  L_K(E)(c-v).$$
\end{proposition}
\begin{proof}
  Since $(c-v)\cdot c^\infty = c^\infty - c^\infty = 0$ in $V_{[c^\infty]}$, we get $
  L_K(E)(c-v) \subseteq {\rm Ker}(\rho_{c^\infty})$.  We now proceed to show   
 that $\Ker (\rho_{c^\infty}) \subseteq 
 L_K(E)(c-v)$.   For notational convenience we denote the left ideal $ L_K(E)(c-v)$ of $L_K(E)$ by $J$.  
 

So let $\lambda \in \Ker (\rho_{c^\infty})$, and   write 
$$\lambda = \sum_{\mu \in \M}k_\mu \mu$$
 where $\M \subseteq \mathcal{S}$ is some finite set of distinct standard form monomials in $L_K(E)$, and $k_\mu \in K^\times$.   By Lemma \ref{M_1inR(c-1)Lemma} we may assume that $\M \subseteq \mathcal{S}_2$; that is, by Lemma \ref{S_1andS_2description}, we may assume that, for each $\mu \in \M$, $\mu = \alpha_\mu c_{i_\mu}^*(c^*)^{n_\mu}$ for some path $\alpha_\mu$, some $0\leq i_\mu \leq t$, and some $n_\mu \geq 0$.  

So we have $\lambda = \sum_{\mu \in \M}k_\mu \alpha_\mu c_{i_\mu}^*(c^*)^{n_\mu}$.   By hypothesis $\lambda \cdot c^\infty = 0$ in $V_{[c^\infty]}$, so that
 $$ \sum_{\mu \in \M}k_\mu \alpha_\mu c_{i_\mu}^*(c^*)^{n_\mu}\cdot c^\infty = 0 \ \  \mbox{in} \ V_{[c^\infty]}.$$
 But $(c^*)^{n}\cdot c^\infty = c^\infty$ in $V_{[c^\infty]}$ for any $n\in \mathbb{Z}$.  So
  $$ \sum_{\mu \in \M}k_\mu \alpha_\mu c_{i_\mu}^*\cdot c^\infty = 0 \ \  \mbox{in} \ V_{[c^\infty]}.$$
 Also, $c_i^* \cdot c^\infty = d_i \cdot c^\infty$ in $V_{[c^\infty]}$ for any $0\leq i \leq t$.   So
  $$ \sum_{\mu \in \M}k_\mu \alpha_\mu d_{i_\mu} \cdot c^\infty = 0 \ \  \mbox{in} \ V_{[c^\infty]}.$$
  Now define 
  $$\lambda^\prime = \sum_{\mu \in \M}k_\mu \alpha_\mu d_{i_\mu}.$$
Then the previous equation gives that $\lambda^\prime \in {\rm Ker}(\rho_{c^\infty})$.  

  We claim that $\lambda \in J$ if and only if $\lambda^\prime \in J$.   To show this, we show that $\overline{\lambda} = \overline{\lambda^\prime}$ as elements of $L_K(E) / J$.     
  We note first that $\overline{c_i^*} = \overline{d_i}$ in $L_K(E)/J$; this follows immediately from the observation that $d_i - c_i^* = c_i^*(c-v)\in J$.    But then in $L_K(E)/J$ we have 
\begin{align*}
  \overline{\lambda}  &=  \overline{\sum_{\mu \in \M}k_\mu \alpha_{\mu} c_{i_\mu}^*(c^*)^{n_\mu}}  \\
 &=\sum_{\mu \in \M}k_\mu \alpha_{\mu} c_{i_\mu}^*\overline{(c^*)^{n_\mu}} \\ 
 & =\sum_{\mu \in \M}k_\mu \alpha_{\mu} c_{i_\mu}^*\overline{v}  \ \ \ \ \ \ \ \ \mbox{by Lemma \ref{M_1inR(c-1)Lemma}(1)} \\
 &  =\sum_{\mu \in \M}k_\mu \alpha_{\mu} \overline{c_{i_\mu}^*} \\
 &   =\sum_{\mu \in \M}k_\mu \alpha_{\mu} \overline{d_{i_\mu}}    \ \ \ \ \ \ \ \ \ \ \mbox{by the above note} \\
  &   = \overline{\sum_{\mu \in \M}k_\mu \alpha_{\mu} d_{i_\mu}} \\
  & = \overline{\lambda^\prime}.\\
\end{align*}
Thus in order to show that $\lambda \in J$, it suffices to show that $\lambda^\prime \in J$, i.e., that $\overline{\lambda^\prime} = \overline{0}$ in $L_K(E)/J$.   But $\lambda^\prime \in {\rm Ker}(\rho_{c^\infty})$, i.e., $\sum_{\mu \in \M}k_\mu \alpha_{\mu} d_{i_\mu} \cdot c^\infty = 0$ in $V_{[c^\infty]}$.  
 Now partition $\M = \sqcup_{t=1}^\ell \M_t$ in such a way that $\mu \sim \mu^\prime \in \mathcal{M}_t$ (for some $t$)  if and only if  $\alpha_{\mu} d_{i_\mu} \cdot c^\infty = \alpha_{\mu^\prime} d_{i_{\mu^\prime}} \cdot c^\infty$   in $V_{[c^\infty]}$.  
 
 By Lemma \ref{pcinftyequalsqcinftyLemma}, if $\mu \sim \mu^\prime$      then  
 $\overline{\alpha_{\mu} d_{i_\mu}} = \overline{\alpha_{\mu^\prime} d_{i_{\mu^\prime}}}$
  in $L_K(E)/J$; we denote this common element of $L_K(E)/J$ by $\overline{x_t}$. 
 
 Now
 $\sum_{\mu \in \M}k_\mu \alpha_{\mu} d_{i_\mu} \cdot c^\infty = 0$ gives
 $$\sum_{t=1}^\ell \sum_{\mu \in \M_t}k_\mu \alpha_{\mu} d_{i_\mu} \cdot c^\infty = 0,$$
 which by the linear independence of sets of distinct elements of the form $\alpha \cdot c^\infty$ in $V_{[c^\infty]}$ gives  $\sum_{\mu \in \M_t}k_\mu = 0$ for each $1\leq t \leq \ell$.  But then 
\begin{align*} 
\overline{\lambda^\prime} &  = \overline{\sum_{\mu \in \M}k_\mu \alpha_{\mu} d_{i_\mu}}  = \overline{\sum_{t=1}^\ell \sum_{\mu \in \M_t}k_\mu \alpha_{\mu} d_{i_\mu}} \\
& =   \sum_{t=1}^\ell \sum_{\mu \in \M_t}k_\mu \overline{\alpha_{\mu} d_{i_\mu}} =  \sum_{t=1}^\ell \sum_{\mu \in \M_t}k_\mu \overline{x_t} \\
& =\sum_{t=1}^\ell (\sum_{\mu \in \M_t}k_\mu) \overline{x_t} =\sum_{t=1}^\ell (0) \overline{x_t} = \overline{0}, \\
 \end{align*}
 which establishes that   ${\rm Ker}(\rho_{c^\infty}) \subseteq  
 L_K(E)(c-v)$, as desired.   The claim about $\hat\rho_{c^\infty}$ follows easily from $L_K(E)=\sum_{u\in E^0\setminus \{v\}}L_K(E)u\oplus L_K(E)v$.
 \end{proof}

\begin{lemma}\label{R(c-1)isoR}  Let $E$ be any graph.
Let $c$ be a  simple closed path in $E$ based at the vertex $v$, and let $r\in L_K(E)v$.  Then  $r(c-v) = 0$ in $L_K(E)$ if and only if $r=0$.  In particular, the map $$\rho_{c-v}: L_K(E)v \rightarrow L_K(E)(c-v)$$ is an isomorphism of left $L_K(E)$-modules. 

Furthermore, if $E$ is a finite graph, then 
 the  map $$\hat{\rho}_{c-1}: L_K(E) \rightarrow L_K(E)(c-1)$$ is an isomorphism of left $L_K(E)$-modules.

\end{lemma}
\begin{proof}  Let $r\in L_K(E)v$.   If $r(c-v) = 0$ then $rc = rv = r$, which recursively gives $rc^j = r$ for any $j \geq 1$.  Now write $r=\sum_{i=1}^{n} k_i{\alpha_i}{\beta_i}^*$, where the $\alpha_i$ and  $\beta_i$ are in ${\rm Path}(E)$.   We note that, for any $m\in \N$,  if  $\beta \in {\rm Path}(E)$   has length at most $ m$,  then $\beta^*$  has the property that  $\beta^* c^m$ is either $0$ or an element of ${\rm Path}(E)$ in  $L_K(E)$.   
Now let $N$ be the maximum length of the paths in the set $\{ \beta_1, \beta_2, \dots, \beta_n\}$.  Then the above discussion shows that   $r c^N$ is an element of $L_K(E)v$ of the form $\sum_{i=1}^{n} k_i \gamma_i$, where  $\gamma_i \in {\rm Path}(E)$ for $1\leq i \leq n$; that is, $r c^N \in KE$.     But $r c^N = r$, so that $r \in KE$.    However,  the equation  $rc = r$  (i.e., $r(c-v) = 0$)    has only the zero  solution in $KE$   by a degree argument.      So $r = 0$.

The second statement is established in an almost identical manner.
\end{proof}


We now have all the tools to describe a projective resolution  for the modules $V_{[c^\infty]}$ where $c$ is a simple closed path, thus completing the study of the second type of Chen simple module.  

\begin{theorem}[Type(2)] \label{projresofVcinftyfromL(E)v}      Let $E$ be any graph.
Let $c$ be a simple closed path in $E$, with $v=s(c)$.  Then the Chen simple module $V_{[c^\infty]}$ is finitely presented.   Indeed, a projective resolution of $V_{[c^\infty]}$ is given by
$$ \xymatrix{ 0 \ar[r] &  L_K(E)v \ar[r]^{\rho_{c-v}} & L_K(E)v \ar[r]^{  \ \ \rho_{c^\infty}} &  V_{[c^\infty]} \ar[r] & 0}.$$
If $E$ is a finite graph, an alternate projective resolution of $V_{[c^\infty]}$ is given by
$$ \xymatrix{ 0 \ar[r] &  L_K(E) \ar[r]^{ \hat\rho_{c-1}} & L_K(E) \ar[r]^{  \ \hat\rho_{c^\infty}} &  V_{[c^\infty]} \ar[r] & 0}.$$
\end{theorem}

\begin{proof}   $V_{[c^\infty]}$ is a simple left $L_K(E)$-module by Theorem \ref{Chentheoremforsimples},   and $c^\infty = vc^\infty $ is a nonzero element in $V_{[c^\infty]}$.  So  the map $\rho_{c^\infty}:L_K(E)v \rightarrow V_{[c^\infty]}
$  is surjective. By  Proposition \ref{Vsubcinftypropforgeneralgraphs} we have  ${\rm Ker}(\rho_{c^\infty}) = L_K(E)(c-v)$. We get the first short exact sequence since by Lemma~\ref{R(c-1)isoR} the map $$\rho_{c-v}: L_K(E)v \rightarrow L_K(E)(c-v)$$ is an isomorphism of left $L_K(E)$-modules. Moreover since $v$ is idempotent, $L_K(E)v$ is a projective left $L_K(E)$-module.

Assume now that $E$ is a finite graph. Let us see that ${\rm Ker}(\hat{\rho}_{c^\infty})=L_K(E)(c-1)$. Since ${\rm Ker}(\hat{\rho}_{c^\infty}) $ clearly contains $u$ for any $u\neq v \in E^0$, we have $c-1 = c - \sum_{u\in E^0}u = (c-v) - \sum_{u\neq v} u\in {\rm Ker}(\hat{\rho}_{c^\infty}) $.    But for any $u\neq v=s(c)$  we have $u c = 0$,  so that $u = - u(c-1) \in L_K(E)(c-1)$. Since   $c-v = v(c-1) \in L_K(E)(c-1)$, using Proposition \ref{Vsubcinftypropforgeneralgraphs} we have shown  that each of the generators of ${\rm Ker}(\hat{\rho}_{c^\infty}) $ is in $L_K(E)(c-1)$.    But by Lemma ~\ref{R(c-1)isoR}, $\hat{\rho}_{c-1}: L_K(E) \rightarrow L_K(E)(c-1)$ is an isomorphism of left $L_K(E)$-modules, thus establishing the result. 
\end{proof}

%

%

\begin{corollary}
Let $E$ be any graph. 
Let $c$ be a simple closed path in $E$, with $v=s(c)$.  Then the Chen simple module $V_{[c^\infty]}$ has projective dimension 1.
\end{corollary}
\begin{proof}
From Theorem~\ref{projresofVcinftyfromL(E)v} we get the exact sequence
$$ \xymatrix{ 0 \ar[r] &  L_K(E)v \ar[r]^{\rho_{c-v}} & L_K(E)v \ar[r]^{  \ \ \rho_{c^\infty}} &  V_{[c^\infty]} \ar[r] & 0}.$$
Since $v$ is an idempotent in $L_K(E)$, the left module $L_K(E)v$ is projective and hence $V_{[c^\infty]}$ has projective dimension $\leq 1$. The left module $V_{[c^\infty]}$ is not projective, otherwise the above sequence splits and $L_K(E)v$ would contain a direct summand isomorphic to $V_{[c^\infty]}$; in particular $L_K(E)v$ would contain a nonzero  element $\alpha$ (the element corresponding to $c^\infty$) such that $c\alpha=\alpha$ and hence $c^n \alpha=\alpha$ for each $n\in \mathbb N$. This is impossible by a degree argument.
\end{proof}

Before we present a projective resolution of the third type of Chen simple module, we study, in the situation where $E$ is row-finite, right multiplication by any of the monomial generators of $V_{[c^\infty]}$ for $c$ a simple closed path or a sink.   We first introduce some notation which will be useful throughout the remainder of the section.

\begin{definition} \label{XsubiDefinitions}
{\rm  Let $E$ be any graph.
Let $\beta = e_1e_2 \cdots e_n$ be a path in $E$.   For each $1\leq i \leq n$  let $\beta_i$ denote $e_1e_2 \cdots e_i$. 
 For each $0\leq i \leq n-1$ let 
$$X_{i}(\beta) = \{f \in E^1 \ | \ s(f) = s(e_{i+1}),  \ \mbox{and} \ f \neq e_{i+1}\}.$$   
The elements of $X_i(\beta)$ are called the {\it exits} of $\beta$ at $s(e_{i+1})$.  Note that, for a given $i$, it is possible that $X_i (\beta) = \emptyset$.      
For each $i \geq 0$ let $J_i(\beta)$ be the left ideal of $L_K(E)$ defined by setting
  $$ J_i(\beta) = \sum_{ f \in X_i(\beta)} L_K(E) f^* \beta_i^*.$$
  (So possibly $J_i(\beta) = \{0\}$, precisely when $X_i(\beta) = \emptyset $.)
When the path $\beta$ is clear from context, we may denote $X_i(\be)$ (resp., $J_i(\be)$) by $X_i$ (resp., $J_i$).

 Now let $p = e_1e_2e_3\cdots \in E^\infty$ be an  infinite path in $E$.   Let $p_0$ denote $s(e_1)$, and for each $i\geq 0$ let $p_{i+1}$ denote $\tau_{\leq {i+1}}(p) = e_1e_2 \cdots e_{i+1}$.   
 For each $i\geq 0$ we define 
$$X_{i}(p) :=  X_i(p_{i+1}), \ \mbox{and} \ J_i(p) := J_i(p_{i+1}).$$
}
    \end{definition}
    
   \begin{definition}\label{F_ibetadef}
 {\rm    Let $E$ be any graph.  Let $\beta = e_1e_2 \cdots e_n$ be a path in $E$ for which no vertex of $\beta$ is an infinite emitter.  For $0\leq i \leq n-1$ let 
    $$F_i(\beta) = \sum_{f \in X_i(\beta)}ff^* \in L_K(E).$$
      Note that this sum is finite by the hypothesis on $\beta$.    (We interpret $F_i(\beta)$ as $0$ in case $X_i(\beta) = \emptyset.$)  In particular, by the (CK2) relation we have 
      $$s(e_{i+1})- F_i(\beta) = e_{i+1}e_{i+1}^*$$
       for $0\leq i \leq n-1$, and by (CK1)  that $F_i(\beta) e_{i+1} = 0.$
       }
      \end{definition}

\begin{lemma}\label{lemma:solveforq}
Let $E$ be any graph.
Let $\alpha = e_1e_2 \cdots e_n$ be a path in $E$ for which no vertex of $\alpha$ is an infinite emitter.  Let $\alpha_i$ denote $e_1e_2 \cdots e_i$ for each $1\leq i \leq n$ (so in particular $\alpha = \alpha_n$).  
Suppose $q,x\in L_K(E)$ satisfy the equation $q\alpha = x$ in $L_K(E)$.   Then $$q = x\alpha^* + q \alpha_{n-1}F_{n-1}(\alpha)\alpha_{n-1}^* + \cdots + q\alpha_1F_1(\alpha)\alpha_1^* + q F_0(\alpha).$$
\end{lemma}

\begin{proof}
Multiply both sides of the equation $q\alpha = x$ by $e_n^*$, to get 
 $$xe_n^* = q\alpha e_n^* = q e_1 \cdots e_{n-1} e_ne_n^* =   q e_1 \cdots e_{n-1} (s(e_n) - F_{n-1}(\alpha)) .         $$
 Multiplying the final term and switching sides, this gives
 $$q e_1 \cdots e_{n-1} =  xe_n^*  + q e_1 \cdots e_{n-1} F_{n-1}(\alpha).
       $$
Multiplying now both sides of this displayed equation on the right by $e_{n-1}^*$, and proceeding in the same way, we easily get
 $$q e_1 \cdots e_{n-2} =  xe_n^* e_{n-1}^*  + q e_1 \cdots e_{n-1} F_{n-1}(\alpha)e_{n-1}^* + q e_1 \cdots e_{n-2} F_{n-2}(\alpha).
       $$
       Continuing in this way, after $n$ steps we reach 
       $$q \ = \  xe_n^* e_{n-1}^* \cdots e_1^*  \  + \  q e_1 \cdots e_{n-1} F_{n-1}(\alpha)e_{n-1}^* \cdots e_1^* 
     \  + \ \cdots  \ +  \ qe_1F_1(\alpha)e_1^* \ + \ qF_0(\alpha) $$
\noindent
as desired.
\end{proof}

Of course, if $c$ is a simple closed path or a sink, any nonzero element of the Chen simple module $V_{[c^\infty]}$ generates $V_{[c^\infty]}$; this is in particular true of  any ``monomial" element $\alpha c^\infty$, where $\alpha$ is a path in $E$ for which $r(\alpha) = s(c)$.   We describe here the  projective resolution corresponding to such elements, in case $E$ is row-finite.

\begin{theorem}[Types (1) $\&$ (2)]\label{kernelresultforalphac^infty}
Let $E$ be any  graph.
Let $c$ be a simple closed path or a sink in $E$, with $v=s(c)$.   Let $\alpha = e_1 e_2 \cdots e_n$ be any path in $E$ for which no vertex of $\alpha$ is an infinite emitter, and for which $r(\alpha) = r(e_n)= v$.  Let $u$ denote $s(\alpha) = s(e_1)$.  Then the following is a projective resolution of the  Chen simple $L_K(E)$-module $V_{[c^\infty]}$:
$$ \xymatrix{ 0 \ar[r] &  L_K(E)(\alpha c \alpha^* -u) \ar[r] & L_K(E)u \ar[r]^{  \ \ \rho_{\alpha c^\infty}} &  V_{[c^\infty]} \ar[r] & 0}.$$
\end{theorem}

\begin{proof}
Since $u\alpha = \alpha$, we have that $\rho_{\alpha c^\infty}(u) = \alpha c^\infty$ is a nonzero element of the Chen simple module $V_{[c^\infty]}$, so that $\rho_{\alpha c^\infty}$ is surjective.  So we need only establish that ${\rm Ker}(\rho_{\alpha c^\infty}) = L_K(E)(\alpha c \alpha^* -u) $.   Since $\rho_{\alpha c^\infty}(\alpha c \alpha^* -u) = (\alpha c \alpha^* -u) \alpha c^\infty = \alpha c c^\infty - \alpha c^\infty = 0,$ it remains  only to show that ${\rm Ker}(\rho_{\alpha c^\infty}) \subseteq L_K(E)(\alpha c \alpha^* -u) $.   

So let $q \in {\rm Ker}(\rho_{\alpha c^\infty})$; specifically, $q\alpha c^\infty = 0$.  But then $q\alpha \in {\rm Ker}(\rho_{ c^\infty})$, which, by Theorem \ref{projresofVcinftyfromL(E)v}, is precisely $L_K(E)(c-v)$.  So 
$$q\alpha = r(c-v)$$
 for some $r\in L_K(E)$.   
     By Lemma \ref{lemma:solveforq}, we have 
     $$q = r(c-v)\alpha^* + q \alpha_{n-1}F_{n-1}(\alpha)\alpha_{n-1}^* + \cdots + q\alpha_1F_1(\alpha)\alpha_1^* + q F_0(\alpha).$$
Using  this representation of $q$, it suffices to show that each of the summands on the right hand side is an element of $L_K(E)(\alpha c \alpha^* -u)$.  Since easily we get $(c-v)\alpha^* = \alpha^* (\alpha c \alpha^* - u)$, we have that $ r(c-v)\alpha^* \in L_K(E)(\alpha c \alpha^* -u)$.   But for each $0 \leq i \leq n-1$ we have        $F_i(\alpha) \alpha_i^* \alpha = F_i(\alpha)e_{i+1}\cdots e_n = 0$ (using the observation made in Definition \ref{F_ibetadef}).  Using this, we see  that      $q \alpha_{i}F_{i}(\alpha)\alpha_{i}^* = - q \alpha_{i}F_{i}(\alpha)\alpha_{i}^*(\alpha c \alpha^* - u)$, so that $q \alpha_{i}F_{i}(\alpha)\alpha_{i}^* \in L_K(E)(\alpha c \alpha^* -u)$ for each $0\leq i \leq n-1$, thus completing the proof.
\end{proof}


We now describe a projective resolution of the third type of Chen simple module, namely, one corresponding to an irrational infinite path.  Whereas a  Chen simple corresponding to a rational path is always finitely presented, we will see that the determination of the finite-presentedness of a Chen simple corresponding to an irrational infinite path will depend on the structure of the graph itself.   


\begin{lemma}\label{lemma:x}
Let $E$ be  any  graph.
Let $p$ be an  irrational infinite path in $E$ with $s(p)=v$, and let $\rho_p: L_K(E)v \rightarrow V_{[p]}$ be the map $r \mapsto rp$.  Let  $x\in {\rm Ker}(\rho_p)$.  Then  there exists $n_x\in\mathbb N$ such that $x \tau_{\leq n_x}(p)=0$ in $L_K(E)$.  In other words, if $x p = 0$ in $V_{[p]}$, then $x p_{n_x} = 0$ in $L_K(E)$ for some finite initial segment $p_{n_x}$ of $p$. 
\end{lemma}
\begin{proof}
Let $x=\sum_{i=1}^mk_i\alpha_i\beta^*_i\in {\rm Ker}(\rho_p)$, where $\alpha_i, \beta_i \in {\rm Path}(E)$. Denote by $N$ the maximum length of the $\beta_i$, $i=1,...,m$. We have
\[\rho_p(x)=\sum^m_{i=1} k_i \al_i\rho_p(\be^*_i)=\sum^m_{i=1} k_i \al_i\rho_{\tau_{> N}(p)}(\be^*_i\tau_{\leq N}(p)).\]
Since the length of each $\be_i$ is less than or equal to $N$, $t_i:=\be^*_i\tau_{\leq N}(p)$ is either zero or a real path. Therefore
\[0 = \rho_p(x)=\sum^m_{i=1} k_i \al_i\rho_{\tau_{> N}(p)}(t_i)=\rho_{\tau_{> N}(p)}(\sum^m_{i=1} k_i \al_it_i)=\rho_{\tau_{> N}(p)}(\sum_{\ell=1}^{m'} h_\ell \gamma_\ell)     = \sum_{\ell=1}^{m'} h_\ell \gamma_\ell \tau_{> N}(p)  ,\]
where the $\gamma_\ell$ ($1 \leq \ell \leq m'$)  are distinct elements of the form $\al_it_i$ in ${\rm Path}(E)$, and $h_\ell \in K$.
Since $p$ is irrational and   $\gamma_\ell$ ($1 \leq \ell \leq m'$) are distinct paths, we claim that the infinite paths $\gamma_\ell\tau_{> N}(p)$ ($1 \leq \ell \leq m'$) are distinct elements of $V_{[p]}$, as follows. Assume to the contrary that 
$\gamma_i\tau_{> N}(p)=\gamma_j\tau_{> N}(p)$ for some  $i\not=j$; necessarily $\gamma_i$ and $\gamma_j$ have distinct lengths $s_i$ and $s_j$. Assume $s_i-s_j=s>0$; then 
\[\gamma_i\tau_{> N}(p)=\gamma_j \kappa_i\tau_{> N}(p)=\gamma_j\tau_{> N}(p),\]
and hence $\kappa_i\tau_{> N}(p)=\tau_{> N}(p)$,
where $\kappa_i$ is a suitable element of ${\rm Path}(E)$ having length $s$. Therefore $\tau_{> N}(p)=\tau_{>s}(\tau_{> N}(p))=\tau_{>s+N}(p)$.  But this property implies by Remark \ref{pathsremark} that $p$ is rational, contrary to hypothesis.  Thus the $\gamma_\ell\tau_{> N}(p)$ ($1 \leq \ell \leq m'$)  are distinct infinite paths.  

Consequently,  the set $\{ \gamma_\ell\tau_{> N}(p) \  | \ 1 \leq \ell \leq m'$\}  is  linearly independent over $K$, so the previously displayed equation $0 = \sum_{\ell = 1}^{m'} h_\ell \gamma_\ell\tau_{> N}(p)$  yields that  $h_\ell = 0$ for each $1 \leq \ell \leq m'$. Therefore 
\[x\tau_{\leq N}(p)=\sum_{i=1}^mk_i\al_i\be_i^*\tau_{\leq N}(p)=\sum_{i=1}^mk_i\al_it_i=\sum_{\ell = 1}^{m'} h_\ell \gamma_\ell=0,\]
as desired.  \end{proof}

\begin{lemma}\label{xbetaequalszerogivesxinsumJi}    Let $E$ be any graph.
Suppose $\beta$ is a path of length $n$ in $E$ for which no vertex of $\beta$ is an infinite emitter, and for which $s(\beta)=v$.  For each $0\leq i \leq n-1$ let $J_i(\beta)$ be the left ideal of $L_K(E)$ given in Definition \ref{XsubiDefinitions}.    If $x \in L_K(E)v$ has $x\beta = 0$, then $x\in  \sum_{i=0}^{n-1} J_i (\beta)$.  
\end{lemma}


\begin{proof}   Write $\beta = e_1e_2 \cdots e_n$.   For each $1\leq i \leq n$ let $\beta_i = e_1 e_2 \cdots e_i$.  So $\beta = \beta_n$, and thus by  hypothesis we are assuming that $x\beta_n = 0$.  Then using
the (CK2) relation  at the vertices $s(e_1), s(e_2), \cdots, s(e_n)$ in order (this is possible by the hypothesis on $\beta$), and interpreting empty sums as $0$, we get 
\begin{align*}
x  & =  xv    =  x( \sum_{f\in X_0(\beta)} ff^* + \beta_1\beta_1^*)  =   x( \sum_{f\in X_0(\beta)} ff^*) + x\beta_1r(\beta_1)\beta_1^*\\
& = x( \sum_{f\in X_0(\beta)} ff^*) + x\beta_1(\sum_{f\in X_1(\beta)} ff^* + e_2e_2^*)\beta_1^* \\
& =  x( \sum_{f\in X_0(\beta)} ff^*) + x\beta_1(\sum_{f\in X_1(\beta)} ff^*)\beta_1^* + x\beta_2r(\beta_2)\beta_2^* \\
& = \cdots \\
& = x( \sum_{f\in X_0(\beta)} ff^*) + x\beta_1(\sum_{f\in X_1(\beta)} ff^*)\beta_1^* +   \cdots  
 + x\beta_{n-1}(\sum_{f\in X_{n-1}(\beta)} ff^*)\beta_{n-1}^* +  x\beta_{n-1}e_ne_n^*\beta_{n-1}^* \\
& =  x( \sum_{f\in X_0(\beta)} ff^*) + x\beta_1(\sum_{f\in X_1(\beta)} ff^*)\beta_1^* +   \cdots  + x\beta_{n-1}(\sum_{f\in X_{n-1}(\beta)} ff^*)\beta_{n-1}^* +  x\beta_n\beta_n^* \\
& =  x( \sum_{f\in X_0(\beta)} ff^*) + x\beta_1(\sum_{f\in X_1(\beta)} ff^*)\beta_1^* +   \cdots  + x\beta_{n-1}(\sum_{f\in X_{n-1}(\beta)} ff^*)\beta_{n-1}^* +  0,
 \end{align*} 
with the final statement following from the hypothesis that $x\beta = x\beta_n = 0$. Thus 
$$ x =  \sum_{f\in X_0(\beta)}( x f )f^* +\sum_{f\in X_1(\beta)} ( x\beta_1f) f^*\beta_1^* +   \cdots  +\sum_{f\in X_{n-1}(\beta)}  (x\beta_{n-1}f) f^*\beta_{n-1}^* \in  \sum_{i=0}^{n-1} J_i (\beta).$$
\end{proof}

\begin{lemma} \label{Xinonempty}   Let $E$ be a finite graph, and let $p = e_1e_2\cdots \in E^\infty$ be an  irrational infinite path in $E$.   Then    $X_i(p)$ is nonempty  for infinitely many $i\in \Z^+$.    Consequently, in this case, $J_i(p)$ is nonzero for infinitely many $i \in \Z^+$.  
\end{lemma}
\begin{proof}
Suppose to the contrary that there exists $N \in \N$ for which $X_i(p) = \emptyset$ for all $i\geq N$.    Since $E^0$ is finite, there exist $t, t' \geq N$, $t < t'$,  for which $s(e_t) = s(e_{t'})$.   But $X_t (p)= \emptyset$ then gives $e_t = e_{t'}$, and in a similar manner yields $e_{t+\ell} = e_{t'+\ell}$  for all $\ell \in \Z^+$.    If $d$ denotes the closed path $e_t e_{t+1} \cdots e_{t'-1}$, then we get $p \sim d^\infty$, the desired contradiction.
\end{proof}

We note that  Lemma \ref{Xinonempty} is not necessarily true without the finiteness hypothesis on the graph.  For instance, let  $M_{\N}$ be the graph    
$$  \xymatrix{  \bullet \ar[r]^{e_1} & \bullet \ar[r]^{e_2} &\bullet \ar[r]^{e_3} &  \cdots}$$
and let $p \in M_\N^\infty$ be the  irrational infinite path $e_1e_2 \cdots$.  Then $X_i(p) = \emptyset$ for all $i\geq 0$.

\begin{corollary}\label{KernelissumofJsubi}
Let $E$ be any graph.
Let $p \in E^\infty$ be an  irrational infinite path in $E$ for which no vertex of $p$ is an infinite emitter, and for which $s(p)=v$.  Let $\rho_p : L_K(E)v \rightarrow V_{[p]}$ be the map $r \mapsto rp$. For each $i\geq 0$ let $J_i(p)$ be the left ideal of $L_K(E)$ given in Definition  \ref{XsubiDefinitions}.   Then   $${\rm Ker}(\rho_p) = {\bigoplus}_{i=0}^\infty J_i(p).$$
\end{corollary}
\begin{proof}   Clearly, for every $i \geq 0$, each element of $J_i(p)$ is in ${\rm Ker}(\rho_p)$.  Now  suppose $x\in L_K(E)$ has $xp = 0$ in $V_{[p]}$.  By Lemma \ref{lemma:x}, $x\beta = 0$ where $\beta = p_n = \tau_{\leq n}(p)$ for some $n\in \N$.   Then Lemma \ref{xbetaequalszerogivesxinsumJi}, together with the definition of $J_i(p)$ for $p\in E^\infty$, gives that ${\rm Ker}(\rho_p) = {\sum}_{i=0}^\infty J_i(p).$

 Now suppose $  \sum_{i=0}^n r_i = 0$ in $L_K(E)$, where $r_i \in J_i(p)$ for $0 \leq i \leq n$.  By construction, $r_i p_n = 0$ for all $i < n$. On the other hand, for any $f\in X_n(p), f^* p_n^* p_n p_n^* = f^* p_n^*$, so that $r_n p_n p_n^* = r_n$ for all $r_n \in J_n(p)$.   Thus  multiplying both sides of the proposed equation $  \sum_{i=0}^n r_i = 0$  on the right by $p_n p_n^*$ gives $ r_n = 0$.   Using this same idea iteratively, we get  $r_i = 0$ for all $0\leq i \leq n$,   so that the sum is indeed direct. 
\end{proof}

We note that Corollary \ref{KernelissumofJsubi} is not necessarily true without the finite emitter hypothesis on the vertices of $p$.   For instance, let $F$ be the graph
$$    \xymatrix{ \bullet^w &   \bullet^v  \ar@{=>}[l]_\infty  \ar[r]^{e_1} & \bullet \ar[r]^{e_2} &\bullet \ar[r]^{e_3} &  \cdots} $$
where there are infinitely many edges $\{f_i \ | \ i \in \Z^+\}$ from $v$ to $w$.   Let $p$ be the  irrational infinite path $e_1e_2 \cdots$, and let $\rho_p: L_K(E)v \rightarrow V_{[p]}$ as usual.   Then easily $ x = v - e_1e_1^* \in {\rm Ker}(\rho_p)$.  However, $x \notin {\sum}_{i=0}^\infty J_i(p)$, since otherwise this would yield that $v$ is a finite sum of $e_1e_1^*$ plus terms of the form $ r_if_if_i^*$ for $r_i \in L_K(E)$, which cannot happen as $v$ is an infinite emitter.

\begin{remark} 
{\rm  Corollary \ref{KernelissumofJsubi}  shows that if $E$ is row-finite and  $p$ is an irrational infinite path, then  ${\rm Ker}(\rho_p)$ is generated by those ghost paths of $L_K(E)$ which annihilate some (finite) initial path of $p$.    Effectively, this is the main difference between the rational and irrational cases; in the rational case, where $p = d^\infty$ and $s(p)=v$, there are additional elements in ${\rm Ker}(\rho_p)$ which are not of this form, namely, elements of the form $r(d-v)$ where $r\in L_K(E)$. 
}
\end{remark}

\begin{lemma} \label{eachJiisadirectsum}  Let $p$ be an  irrational infinite path in an arbitrary  graph $E$.    For $f\in E^1$ let $v_f$ denote the vertex $r(f)$.  Then, for each $i\geq 0$,     
$$J_i (p) \ =  \ \bigoplus_{f\in X_i(p)}L_K(E)  f^* p_i^*  \  \cong \ \bigoplus_{f\in X_i(p)}L_K(E)v_f.$$
as left $L_K(E)$-modules.    In particular, each $J_i(p)$ is a projective left $L_K(E)$-module.     
 \end{lemma}

\begin{proof} By definition  $J_i (p) = \sum_{f\in X_i(p)}L_K(E)  f^* p_i^*$.   We claim the sum is direct.  So suppose $0 = \sum_{f\in X_i(p)}r_f f^*p_i^* $, with $r_f\in L_K(E)$ for each $f\in X_i(p)$.   Without loss we may assume that each expression $r_f f^*$ is nonzero, so that we may further  assume without loss that $r_f v_f = r_f$ for each $f\in X_i(p)$.   Take  $g\in X_i(p)$; by multiplying  $0 = \sum_{f\in X_i(p)}r_f f^*p_i^* $ on the right by  $p_ig$, and using the (CK1) relation, we get  $0= r_g g^*g = r_g\cdot v_g = r_g$.  Thus the sum is direct, so that $J_i (p) = \oplus_{f\in X_i(p)}L_K(E)  f^* p_i^*$.      But for $g\in X_i(p)$ it is easy to show that $L_K(E) g^* p_i^* \cong L_K(E)v_g$, by the map $x \mapsto x p_ig$.    
\end{proof}

\begin{theorem}[Type (3)] \label{Vpinftynotfinitelypresented}  Let $E$ be any graph.
Let $p \in E^\infty$ be an  irrational infinite path in $E$ for which no vertex of $p$ is an infinite emitter.   Then the Chen simple $L_K(E)$-module $V_{[p]}$ is 
finitely presented if and only if $X_i(p)$ is nonempty only for 
finitely many $i\in \Z^+$.  

In particular, if $E$ is a finite graph, then $V_{[p]}$ is not finitely presented. 
 \end{theorem}
 \begin{proof}  Let $v$ denote $s(p)$.   We consider the exact sequence
 $$ \xymatrix{ 0 \ar[r] &  {\rm Ker}(\rho_p) \ar[r] & L_K(E)v \ar[r]^{ \ \ \ \rho_p} &  V_{[p]} \ar[r] & 0}.$$
 By Corollary \ref{KernelissumofJsubi}  we have that ${\rm Ker}(\rho_p) = \oplus_{i=0}^\infty J_i(p).$  Furthermore, each $J_i(p)$ is projective by Lemma \ref{eachJiisadirectsum}, so the given exact sequence is a projective resolution of $V_{[p]}$.    Therefore $V_{[p]}$ is finitely presented if and only if $J_i(p)$
 is nonzero only for finitely many $i\in \Z^+$, i.e. $X_i(p)$ is nonempty only for finitely many $i\in \Z^+$.
 
 For the particular case, when $E$ is finite then by Lemma \ref{Xinonempty} $J_i(p)$ is nonzero for infinitely many $i$.   
  \end{proof}
 
 \begin{corollary}\label{cor:projdim1}
 If $E$ is a finite graph, and $p \in E^\infty$ is an  irrational infinite path in $E$, then the Chen simple $L_K(E)$-module $V_{[p]}$ has projective dimension 1.
 \end{corollary}
 \begin{proof}
From Corollary~\ref{KernelissumofJsubi} we get the exact sequence
$$ \xymatrix{ 0 \ar[r] &  {\bigoplus}_{i=0}^\infty J_i(p) \ar[r]& L_K(E)v \ar[r]^{  \ \ \rho_{p}} &  V_{[p]} \ar[r] & 0}.$$ 
Since $v$ is an idempotent in $L_K(E)$, the left module $L_K(E)v$ is projective; by Lemma~\ref{eachJiisadirectsum} also ${\bigoplus}_{i=0}^\infty J_i(p)$ is projective and hence $V_{[p]}$ has projective dimension $\leq 1$. Since $E$ is finite, $J_i(p)$ is not zero for infinitely many $i$ and hence ${\bigoplus}_{i=0}^\infty J_i(p)$ is not finitely generated. Then the left module $V_{[p]}$ is not projective, otherwise ${\bigoplus}_{i=0}^\infty J_i(p)$ would be a not finitely generated direct summand of a cyclic module: contradiction.
 \end{proof}
 
 \begin{remark}\label{remark:projectiveChensimple}
 Let $M_{\mathbb N}$ be the graph
 $$  \xymatrix{  \bullet \ar[r]^{e_1} & \bullet \ar[r]^{e_2} &\bullet \ar[r]^{e_3} &  \cdots}$$
considered previously, and let $p \in M_\N^\infty$ be the  irrational infinite path $e_1e_2e_3 \cdots$.   Then $X_i(p) = \emptyset$ for all $i\geq 0$.  So  by Corollary~\ref{KernelissumofJsubi}, the Chen simple module $V_{[p]}$ is isomorphic to $L_K(E)v$, and hence it is projective.
\end{remark} 

\begin{example}\label{R2irrationalexample}
{\rm We reconsider the graph  $R_2$ and irrational infinite path $q = efeffefffe\cdots \in R_2^\infty$ described in Example \ref{R2Example}.   Then, as $R_2$ is finite, Theorem \ref{Vpinftynotfinitelypresented} yields that the Chen simple module $V_{[q]}$ is not finitely presented.  
} 
\end{example}







\begin{remark}\label{PereRangaresult}
{\rm We note that Theorems \ref{projresofVcinftyfromL(E)v} and \ref{Vpinftynotfinitelypresented} strengthen and sharpen \cite[Proposition 4.1]{AR}, most notably because we have been able to explicitly describe a projective resolution of each of the Chen simple modules.  
}
\end{remark}

%
%
%
%

 In \cite[Theorem~4.12]{AMMS} it is shown that for any graph $E$,  for
any vertex $v\in E^0$,  $L(E)v$ is a simple left ideal if and only if $v$ is a
\emph{line point}, i.e. in the full subgraph of $E$ generated by $\{ u \in E^0 \  |
\text{ there is a path from $v$ to $u$}\}$  there are no cycles, and there are no
vertices which emit more than one edge.     
Our results allow us to recover  \cite[Theorem~4.12]{AMMS}, as follows. 

\begin{corollary}\label{linepointtheorem}
Let $E$ be any graph.
Let $u\in E^0$.  Then $L_K(E)u$ is simple if and only if $u$ is a line point. 
\end{corollary}
\begin{proof}
There are three possibilities:
\begin{enumerate}
\item there is a path $\alpha \in {\rm Path}(E)$ with $s(\alpha)=u$ and for which $r(\alpha) = w$ is a sink in $E$;
\item there is a path $\alpha \in {\rm Path}(E)$ with $s(\alpha)=u$ and for which $r(\alpha) = v$ is the source of a simple closed path $c$;
\item there is an infinite irrational path $q$ for which $s(q) = u$.
\end{enumerate}

If in $\alpha$ (cases 1 and 2) or in $q$ (case 3) there is an infinite emitter $x$, then $L_K(E)x$ is not a simple submodule of $L_K(E)v$ (see \cite[Lemma~4.3]{AMMS}). Therefore we can assume that $\alpha$ (cases 1 and 2) and $q$ (case 3) have no infinite emitter.\\

Cases 1 and 2. By Theorem~\ref{kernelresultforalphac^infty} $L_K(E)u$ is a simple module if and only if $\alpha c\alpha^*=u$, where $c$ is either a simple closed path or a sink. If $c$ is a simple closed path, by a degree argument $\alpha c\alpha^*$ is not a vertex. If $c$ is a sink and $\alpha=e_1 \cdots e_\ell$ then $\alpha c\alpha^*=\alpha\alpha^*=e_1 \cdots e_\ell e_\ell^* \cdots e_1^*$ is equal to $u$ if and only if $e_ie_i^*=s(e_i)$ for $i=1,...,\ell$, i.e. if and only if $s(e_i)$ is the source of only one edge, i.e. $u$ is a line point.\\

Case 3. By Corollary~\ref{KernelissumofJsubi}, $L_K(E)u$ is simple if and only if $\bigoplus_{i=0}^{\infty}J_i(p)=0$ and the latter is equivalent to $p$ having no exits, i.e. $u$ is a line point. 
%
%
\end{proof}

%
%
%
%

\section{Extensions of Chen simple modules}\label{ExtSection}

In this section we use the results of Section \ref{projresSection} to describe ${\rm Ext}^1_{L_K(E)}(S,T)$, where $S$ and $T$ are Chen simple modules over the Leavitt path algebra $L_K(E)$ corresponding to a finite graph $E$.   As a consequence, this will allow us to (among other things) construct classes of indecomposable non-simple $L_K(E)$-modules.

We give here a short review of ${\rm Ext}^1$; see e.g. \cite{W}  for more information.      Let  $R$ be a ring, and let $M,N$ be left $R$-modules.  Suppose
$$ \xymatrix{ 0  \ar[r] &  Q \ar[r]^{   \ \mu}  & P  \ar[r]^{   \ f} &  M \ar[r] & 0}$$
is a short exact sequence with $P$ projective.   Then there is an exact sequence of abelian groups
$$ \xymatrix{   {\rm Hom}_R(P,N)  \ar[r]^{   \ \mu_*  }  & {\rm Hom}_R(Q,N)  \ar[r]^{   \ \Delta_f}  & {\rm Ext}_R^1( M,N) \ar[r] & 0,}$$
where $\mu_*(\varphi) = \varphi \circ \mu$ for $\varphi \in {\rm Hom}_R(P,N),$ and $\Delta_f$ is the ``connecting morphism".   If $\mu$ is viewed as an inclusion of submodules, then $\mu_*(\varphi) = \varphi |_Q$, the restriction of $\varphi$ to $Q$.   Exactness yields that ${\rm Ext}_R^1( M,N) = 0$ if and only $\mu_*$ is surjective. Moreover, ${\rm Ext}_R^1( M,N)\neq 0$ if and only if there exists a non-splitting short exact sequence
$$0\to N \to L \to M\to 0,$$
i.e., $L$ is a non-trivial extension of $N$ by $M$.
For instance  if $M,N$ are simple left $R$-modules, then  ${\rm Ext}_R^1( M,N)\neq 0 $ if and only if there exist   indecomposable left $R$-modules of length 2 which are extensions of $N$ by $M$. 
Finally,  observe that if $R$ is a $K$-algebra over a field $K$, then the abelian group  ${\rm Ext}_R^1( M,N)$ has a natural structure of $K$-vector space for any left $R$-module $M$  and $N$. 

\medskip

We outline our approach.  There are three types of Chen simple modules: those of the form $V_{[w^{\infty}]}$ for a sink $w$; of the form $V_{[c^\infty]}$ for a simple closed path $c$; and of the form $V_{[p]}$ for an irrational infinite path $p$.    Let $T$ denote any Chen simple module.    In Lemma \ref{Extforsinks} we make the (trivial) observation that ${\rm Ext}_{L_K(E)}^1( V_{[w^{\infty}]},T) = 0$;   in Theorem \ref{dimExtSTprop} we describe ${\rm Ext}_{L_K(E)}^1( V_{[c^\infty]},T)$; and in Theorem \ref{Ext(ST)Sirrational} we describe   ${\rm Ext}_{L_K(E)}^1( V_{[p]},T)$. We recall that we are assuming $w=w^{\infty}\in E^{\infty}$ for any sink $w$.

\begin{lemma}\label{Extforsinks} {\rm (Type (1)) } 
Let $E$ be any graph.
Let $w$ be a sink in $E$, and let $T$ denote any left $L_K(E)$-module. Then ${\rm Ext}_{L_K(E)}^1( V_{[w^{\infty}]},T) = 0$, i.e. any extension of $V_{[w^{\infty}]}$ by $T$ splits.
\end{lemma}
\begin{proof}  This follows immediately from the fact that $V_{[w]} \cong L_K(E)w$ is a projective  $L_K(E)$-module (see Proposition~\ref{projressink}).

\end{proof}

\begin{definition}
{\rm   
Let $T$ be a Chen simple module.   Denote by $U(T)$ the set
$$U(T) := \{v\in E^0 \ | \ vT \neq \{0\} \} =  \{v\in E^0 \ | \ \mbox{there exists } t\in T \ \mbox{with } vt\neq 0\}.$$
}
\end{definition}

\begin{remark}\label{U(T)remark} 
{\rm Let $T$ be a Chen simple module and let $q=e_1 e_2 \cdots \in E^{\infty}$ such that $T=V_{[q]}$. Then   $U(T)$ consists of those vertices $v$ for which there is a path $\alpha\in   {\rm Path}(E)$  having $s(\alpha)=v$ and $r(\alpha)=s(e_i)$ for some $i\geq 1$. Equivalently, a vertex $v\in U(T)$ if and only if there is an infinite path  tail-equivalent to $q$ starting from $v$.
Hence $U(T)$ is a feature of the Chen simple module $T$ that can be read directly from the graph $E$. }
\end{remark}

\begin{definitions}\label{Extdefinitions}
{\rm  Let $E$ be any  graph and let $d$ be a simple closed path in $E$. 

\noindent 

For any $p\in E^\infty$  we say that   $p$  is \emph{divisible by  $d$} if $p=dp'$ for some $p' \in E^\infty$. 

\noindent 

For  any $q \in E^\infty$, we define the set 
 $${L_{(d, q)}}:=\{ p \in E^\infty \ | \  p \sim  q, \ s(p)=s(d) , \text{ and } p \ \text{is not divisible by } d\} \subseteq V_{[q]},$$
 where $V_{[q]}$ is the  Chen simple $L_K(E)$-module generated by $q$.}
 \end{definitions}
 
An infinite path $p$ is divisible by a simple closed path $d\in E$ if and only if $d=t_{\leq \ell}(p)$, where $\ell$ is the length of $d$.
The set  $L_{(d,q)}$ consists of those infinite paths  which start at $s(d)$, and which eventually equal some tail of $q$, but do not start out by traversing the closed path $d$.     Observe that the subset  $L_{(d,q)}$  of $V_{[q]}$  does not depend on $q$ but only on the equivalence class $[q]$. 
Let $T=V_{[q]}$; if $q$ is not tail equivalent to $d^\infty$, then there exists $q'\sim q$ such that $d\not|q'$ and hence $T$ has a generator not divisible by $d$.

\begin{remark}\label{uniquenesslemma}
{\rm Let $d$ be a simple closed path in $E$ and $q\in E^\infty$.

1) Suppose $q$ is not tail equivalent to $d^{\infty}$ and consider the Chen simple module $V_{[q]}$; we can assume without loss of generality that $q$ is not divisible by $d$. The set $L_{(d, q)}$ is not empty if and only if $s(d)$ belongs to $U(\Vq)$; in such a case any $0\not=t\in V_{[q]}$ for which $s(d)t=t$ is a linear combination of infinite paths tail equivalent to $q$ whose sources coincide with $s(d)$. In particular, taking in account the divisibility by $d$ of these infinite paths, $t$ can be written in a unique way as 
 $$t=t_0+dt_1+d^2t_2+\dots+d^st_s,$$ where the $t_i$ are $K$-linear combinations of elements in $L_{(d, q)}$ and $t_s\neq 0$. We call $s\geq 0$ the \textbf{$d$-degree} of $t$ and we denote it by $\deg_d(t)$.
 
 (2) Suppose  $q=d^{\infty}$.  Then  $L_{(d, d^{\infty})}\neq \emptyset$ if and only if there exists a cycle $c\neq d$ with $s(c)=s(d)$.   
Any $0\not=t\in V_{[d^\infty]}$ for which $s(d)t=t$ can be written in a unique way as 
 $$t=kd^{\infty}+t_0+dt_1+d^2t_2+\dots+d^st_s,$$ where the $t_i\in \Vd$ are $K$-linear combinations of elements in $L_{(d, d^{\infty})}$ and $t_s\neq 0$. We call $s\geq 0$ the \textbf{$d$-degree} of $t$ and we denote it by $\deg_d(t)$.
 
In particular, any $0 \neq t \in L_{(d,q)}$ has $d$-degree equal to 0.  We emphasize that, in case $q=d^{\infty}$,   the $d$-degree of the element $d^\infty$ of $\Vd$ is zero too: $\deg_d(d^{\infty})=0$. The $d$-degree is not defined on 0.}
\end{remark}

  \begin{example}\label{L(d,q)example} 
  {\rm  We revisit the graph $R_2$ given by
$$\xymatrix{\bullet^v \ar@(ul,dl)_e \ar@(ur,dr)^f} \ .$$
  Consider the simple closed path $e$ and the rational infinite path $f^\infty$.     Then  $L_{(e,f^\infty)}=\{p\in R_2^{\infty} \ | \ p\sim f^{\infty} \ \text{and} \ p \ \text{is not divisible by} \ e\} \subseteq V_{[f^\infty]}$ contains, for instance, the infinite paths $\{ f ^i e^j f^\infty \ | \ i\geq 1,  j\geq 0\}$.   (There are additional elements of $L_{(e,f^\infty)} $, for instance, $fefef^\infty$.)       
  Moreover, consider an element of the form $ e^j f ^i e f^\infty  \in V_{[f^\infty]}$, with $i\geq 1$ and $j\geq 0$.  Then  ${\rm deg}_e(e^j f^i e f^\infty) = j$.
  
  On the other hand, $L_{(f,f^\infty)}= \{p\in R_2^{\infty} \ | \ p\sim f^{\infty} \ \text{and} \ p \ \text{is not divisible by} \ f\}$ contains the infinite paths $\{ e^i f^\infty \ | \ i\geq 1\}$.   Note that the element $f^\infty$ of $V_{f^\infty}$ is defined to have  ${\rm deg}_f(f^\infty) = 0$.  
    }
  \end{example}
  
Recall that $L_K(E)$ is a ring wit unity if and only if $E$ is finite.
 
 \begin{lemma}\label{equationswithsolutionslemma}
Let $E$ be a finite graph.   Let $d$ be a simple closed path  and $q\in E^\infty$.  Let $t\in V_{[q]}$, and consider the equation in the variable $X$
$$ (d-1)X=t.$$
  The equation admits a solution in $\Vq$ if one  of the following holds: \begin{enumerate}
\item  $s(d)t=0$;
\item  $t=d^np-p$ for some  $p \in L_{(d, q)}$ and  $n\geq 0$.  
\end{enumerate}
\end{lemma} 
\begin{proof}
(1) is easy,  since if $s(d)t=0$ then  $dt=0$, and hence $X=-t$ is a solution.   (2) is nearly as easy, since  we have $(d-1)\sum_{i=0}^{n-1}d^i p=d^n p - p=t$, and hence $X= \sum_{i=0}^{n-1}d^i p$ is a solution.  
 \end{proof}

\begin{lemma}\label{equationswithNOsolutionslemma}
Let $E$ be a finite graph.   Let $d$ be a simple closed path   and let $q\in E^\infty$. Assume either $q=d^\infty$ or $q$ is a generator of  $V_{[q]}$ not divisible by $d$. Let $ 0 \neq t \in V_{[q]}$, and consider the equation 
$$ (d-1)X=t.$$
  Assume  $ t=d^nt'$  for some  $n\geq 0$ and some  $0\neq t'\in \Vq$ for which  $s(d)t' = t'$ and $\deg_d(t')=0$.   Then    the equation has no solution in $\Vq$.  
In particular:

(1)  the equation $(d-1)X = t$ has no solution in $V_{[q]}$ whenever $t\in L_{(d,q)}$, and 

(2)  the equation $(d-1)X = d^\infty$ has no solution in $V_{[d^\infty]}$.

\end{lemma}
\begin{proof}
Let $v=s(d)$. Since $vd = d$ and $t=d^nt'$, we get $vt=t$.  So if $x$ is a solution of $(d-1)X = t$,  then we would have $v(d-1)x = t$, so that   $(d-1)vx = t$; thus we may assume without loss that $vx = x$. Hence the equation yields 
 $$x = dx - t,$$
and, since $t\not=0$, necessarily then $x \neq 0$. Let   $\deg_d(x)=s$  and   write $$x=kd^{\infty}+x_0+dx_1+ \dots +d^sx_s,$$ where the $x_i$'s are linear combination of elements in $L_{(d, q)}$, and $k=0$ in case $q\neq d^{\infty}$. Then, using $d\cdot d^\infty = d^\infty$,   we get 
\begin{align}
t  & = d^nt' = (d-1)x = dx - x  \nonumber \\
  & =kd^{\infty}-kd^{\infty} -x_0+d(x_0-x_1)+\dots+d^s(x_{s-1} -x_{s}) + d^{s+1}x_s \nonumber \\
  &= -x_0+d(x_0-x_1)+\dots+d^s(x_{s-1}-x_{s}) + d^{s+1}x_s.  \nonumber 
  \end{align}
We claim that this is impossible.    Set $x_{-1}=0=x_{s+1}$.   By the uniqueness of the decomposition in Remark~\ref{uniquenesslemma}, since $\deg_d(t')=0$,  one gets $t'=x_{n-1}-x_{n}$ and $x_{i-1}-x_i=0$ for any $i\neq n$, $-1\leq n \leq s+1$.   Then we have $0=x_0=x_1=\dots=x_{n-1}$ and $t'=-x_n=\dots -x_{s+1}=0$, contradiction.

\end{proof}



%
%

Assume  $E$ is a finite graph and $d$ a simple closed path in $E$. In order to compute the groups $\Ext^1_{L_K(E)}(V_{[d^{\infty}]}, T)$ for any Chen simple module $T$,  we can consider the projective resolution of  $V_{[d^{\infty}]}$ 
$$0 \rightarrow {L_K(E)}\stackrel{\hat{\rho}_{(d-1)}}{\to}
 {L_K(E)} \stackrel{\hat{\rho}_{d^{\infty}}}{\to} \Vd \rightarrow 0$$  ensured by Theorem  \ref{projresofVcinftyfromL(E)v}.


\begin{lemma}\label{lemma:equation}   Let $E$ be a finite graph.  Let $d$ be a simple closed path in $E$ and let $T$ be a Chen simple module. Consider the exact sequence
$$ \xymatrix{   {\rm Hom}_{L_K(E)}({L_K(E)},T)  \ar[r]^{  \ \ \  \!  \! \! \!   \! \!  \! \! \!  \hat\rho_{(d-1)_*}} & {\rm Hom}_{L_K(E)}({L_K(E)}, T)  \ar[r]^{ \  \ \  \pi }  & {\rm Ext}_{L_K(E)}^1(\Vd, T) \ar[r] & 0}$$
where $\hat\rho_{(d-1)_*}(\phi)=\phi\circ \hat\rho_{d-1}$,  and $\pi$ is the connecting homomorphism. 
Then $$\pi(\hat\rho_{t})=0 \mbox{ if and only if the equation }   (d-1)X=t  \mbox{ has a solution in } T.$$  Consequently, $\Ext^1_{L_K(E)}(\Vd, T)=0$ 
if and only if $(d-1)X = t$ has a solution in $T$ for every $t\in T$. 
\end{lemma}
\begin{proof}
By  exactness it follows that  $\pi(\hat\rho_{t})=0$ if and only if there exists $x\in T$ such that 
\[\hat\rho_{t}={\hat\rho_{(d-1)_*}}(\hat\rho_{x})=\hat\rho_{x}\circ \hat\rho_{(d-1)}=\hat\rho_{(d-1)x}\]
i.e. if and only if  the equation $(d-1)X = t$ has a solution in $T$.

The final statement follows directly from the exactness of the displayed sequence. 
\end{proof}

\begin{theorem}\label{Ext(ST)Srational} {\rm (Type (2))}  Let $E$ be a finite graph.
Let $d$ be a simple closed path in $E$  and let $T$ be a Chen simple module. 
Then the following are equivalent:
\begin{enumerate}
\item   $\Ext^1_{L_K(E)}(\Vd, T)\neq 0.$
\item   $s(d)\in U(T)$
\end{enumerate}
\end{theorem}
\begin{proof}
(1) $\Rightarrow$ (2)   If $s(d)T= 0$, then   for any $t\in T$ we have $s(d)t=0$, so by Lemma \ref{equationswithsolutionslemma}(1)  the equation $(d-1)X=t$ admits a solution for any $t\in T$.    Applying Lemma~\ref{lemma:equation}(2), we get  that  ${\rm Ext}_{L_K(E)}^1(V_{[d^\infty]}, T) = 0$.

\smallskip

(2) $\Rightarrow$ (1)   First assume  $T\not=V_{[d^{\infty}]}$.  As observed in Remark~\ref{uniquenesslemma}, $T$ admits a generator $q$ not divisible by $d$ and $L_{(d, q)}$ is not empty. Let $p\in L_{(d, q)}$.  
By Lemma \ref{equationswithNOsolutionslemma},  the equation $(d-1)X=p$ has no solution in $\Vq$ and so, by Lemma~\ref{lemma:equation},  $\pi(\hat\rho_{p})\neq 0$.

On the other hand, suppose  $T=V_{[d^{\infty}]}$.  By Lemma \ref{equationswithNOsolutionslemma},  the equation $(d-1)X=d^{\infty}$ has no solution in $V_{[d^\infty]}$, and so, again invoking  Lemma~\ref{lemma:equation},  $\pi(\hat\rho_{d^{\infty}})\neq 0$.

In either case we have established the existence of a nonzero element in ${\rm Ext}^1_{L_K(E)}(\Vd, T).$ 

\smallskip

\end{proof}

\begin{corollary}\label{cor:Ext1VdVd}
Let $E$ be a finite graph.
For any simple closed path $d$, $\Ext^1_{L_K(E)}(\Vd, \Vd)\neq 0$.
\end{corollary}

\begin{example}
{\rm We again revisit the graph $R_2$:  
$$\xymatrix{\bullet^v \ar@(ul,dl)_e \ar@(ur,dr)^f} \ .$$
Let $q$ be any element of $R_2^\infty$.   Let $d$ be any (of the infinitely many)  simple closed paths in $R_2$.   Since  clearly Condition (2) of Theorem \ref{Ext(ST)Srational} is satisfied for $\Vq$, we get that ${\rm Ext}^1_{L_K({R_2})}(\Vd,\Vq) \neq 0$.  \hfill $\Box$

}
\end{example}

\medskip

Having now established conditions which ensure that there exist nontrivial extensions of the Chen simple module $T$ by the simple module $V_{[d^\infty]}$, we now give a more explicit description of the number of such extensions.

\begin{proposition}\label{dimExtSTprop}  Let $E$ be any finite graph.  Let $d$ be a simple closed path in $E$ and let $T$ be a  Chen simple module.  Assume  $q \in E^\infty$  such that $T=\Vq$.  
\begin{enumerate}
\item Suppose $ T \neq \Vd$.  Then $\dim_K \Ext_{L_K(E)}^1(\Vd, T)=|L_{(d, q)}|$.
\item  
 On the other hand, $\dim_K \Ext_{L_K(E)}^1(\Vd,\Vd) =|L_{(d, d^{\infty})}|+1$. 
\end{enumerate}
\end{proposition}
\begin{proof}
We consider the exact sequence
$$ \xymatrix{   {\rm Hom}_{L_K(E)}({L_K(E)},V_{[q]})  \ar[r]^{ \ \ \  \! \!  \! \! \!   \! \!  \! \! \!  \hat\rho_{(d-1)_*}} & {\rm Hom}_{L_K(E)}({L_K(E)},V_{[q]})  \ar[r]^{ \  \ \  \pi }  & {\rm Ext}_{L_K(E)}^1(\Vd,\Vq) \ar[r] & 0}.$$

(1) Without loss of generality we can assume $q$ is not divisible by $d$. By Remark~\ref{uniquenesslemma}(1) and Theorem~\ref{Ext(ST)Srational}, if $L_{(d, q)}=\emptyset$ then $\Ext_{L_K(E)}^1(\Vd, T)=0$. Otherwise,  by Lemmas \ref{equationswithNOsolutionslemma} and  \ref{lemma:equation}, $\pi({\hat\rho_{p}})\neq 0$ for  any path $p\in L_{(d, q)}$. 
We claim that  the set $\{ \pi(\hat\rho_{p}) \ | \ p\in L_{(d, q)}\}$ is  a basis for the  vector space $\Ext_{L_K(E)}^1(\Vd,T)$.
In order to show that $\{ \pi(\hat\rho_{p}) \ | \ p\in L_{(d, q)}\}$ is $K$-linearly independent, suppose there is a $K$-linear combination $0 = k_1 \pi(\hat\rho_{p_1}) + \dots + k_n \pi(\hat\rho_{p_n}) $ in $\Ext_{L_K(E)}^1(\Vd,\Vq)$, with $p_i \in L_{(d,q)}$.   Let $t =k_1p_1+\dots+ k_np_n$ in $\Vq$ so that $\pi(\hat\rho_t)=0$ in $\Ext_{L_K(E)}^1(\Vd,\Vq)$. Thus, applying   Lemma \ref{lemma:equation}, we get that  the equation $(d-1)X=t$ has a solution in $\Vq$. If $t\not=0$, since $s(d)t=t$ and $p_i \in L_{(d,q)}$ we get ${\rm deg}_d(t) = 0$, which is a contradiction by Lemma \ref{equationswithNOsolutionslemma}.  Hence $t=0$ and by the linear independence of $\{p_1,\dots,p_n\}$ in $\Vq$ we get that $k_1 = \dots = k_n = 0.$   So $\{ \pi(\hat\rho_{p}) \ | \ p\in L_{(d, q)}\}$ is $K$-linearly independent.  

We now show that $\{ \pi(\hat\rho_{p}) \ | \ p\in L_{(d, q)}\}$ spans $\Ext_{L_K(E)}^1(\Vd,T)$.  As $\pi$ is surjective, by Lemma \ref{lemma:equation}(1) it suffices to show that any $\pi(\hat\rho_t) \in {\rm Ext}_{L_K(E)}^1(\Vd,T)$ is a $K$-linear combination of elements from this set.  
Write   $t=t'+t''$ where $t'=\sum_{i=1}^{m_{t'}} k_ip'_i$ with $s(p'_i)= s(d)$ and $t''=\sum_{j=1}^{m_{t^{''}}} k_j p^{''}_j$ with $s(p^{''}_j)\neq  s(d)$. By Lemma \ref{equationswithsolutionslemma}(1),  the equation $(d-1)X=t$ has solution in $T=\Vq$ if and only if $(d-1)X=t'$ has solution in $\Vq$, so  we can assume without loss of generality that $s(d)t=t$. Hence $t=t_0+dt_1+d^2t_2+\dots+d^st_s$, and so $\hat\rho_t=\hat\rho_{t_0}+\hat\rho_{dt_1}+\hat\rho_{d^2t_2}+\dots+\hat\rho_{d^st_s}$,  where each $t_i$ is of the form $t_i=\sum_{j=1}^{m_i} k_{ij}u_{ij}$,  for some $u_{ij}\in L(d,q)$. Thus $\pi(\hat\rho_t)=\sum_{j=1}^{m_0}k_{0j}\pi(\hat\rho_{u_{0j}}) + \sum_{j=1}^{m_1}k_{1j}\pi(\hat\rho_{du_{1j}})+\dots+\sum_{j=1}{m_s}k_{sj}\pi(\hat\rho_{d^su_{sj}})$. 
Observe that, by Lemmas \ref{equationswithsolutionslemma}(2) and \ref{lemma:equation}, we get   $\pi(\hat\rho_{d^nu}-\hat\rho_{u})=0$ for any $u\in L_{(d, q)}$ and any $n\in \mathbb{N}$, so $\pi(\hat\rho_{d^nu})=\pi(\hat\rho_{u})$ for any $n\in \mathbb{N}$. Hence   $\{ \pi(\hat\rho_{u}) \ | \ u\in L_{(d, q)}\}$ is  a set of generators for $\Ext_{L_K(E)}^1(\Vd,T)$. 

\smallskip

(2) Let us show that $\{\pi({\hat\rho_{p}}) \ |  \   p\in L_{(d, d^{\infty})}\}\cup \{ \pi(\hat\rho_{d^{\infty}})\}$ is a basis for ${\rm Ext}_{L_K(E)}^1(\Vd,\Vd)$. 
First observe that, by Lemmas \ref{equationswithNOsolutionslemma} and  \ref{lemma:equation}(2), $\pi(\hat\rho_{d^{\infty}})\neq 0$ and  $\pi({\hat\rho_{p}})\neq 0$ for  any  $p\in L_{(d, d^{\infty})}$.  
Arguing as in part (1) we claim  that $\{\pi({\hat\rho_{p}}) \ |  \   p\in L_{(d, d^{\infty})}\}\cup \{ \pi(\hat\rho_{d^{\infty}})\}$ is a  linearly independent set in ${\rm Ext}_{L_K(E)}^1(\Vd,\Vd)$.  Indeed, consider a $K$-linear combination $0 = k_0 \pi(\hat\rho_{d^{\infty}}) +  k_1 \pi({\hat\rho_{p_1}}) + \dots + k_n\pi({\hat\rho_{p_n}})$ in ${\rm Ext}_{L_K(E)}^1(\Vd,\Vd)$.  Define $y =k_0 d^{\infty}+k_1p_1+\dots+ k_np_n \in \Vd$ so that $\pi(\hat\rho_y)=0$ and hence, by Lemma \ref{lemma:equation}(2),  the equation $(d-1)X=y$ has a solution in $\Vd$.  Note that if $y\not=0$ then ${\rm deg}_d(y) = 0$ (whether or not $k_0 = 0$) since each $p_i \in L_{(d,d^\infty)}$, which is a contradiction
 by Lemma \ref{equationswithNOsolutionslemma}.     So $y = 0$, which yields that each $k_i$ ($0 \leq i \leq n$) is $0$.

Since any $t$ in $\Vd$ with $s(d)t=t$ is of the form $t=kd^{\infty}+t_0+dt_1+d^2t_2+\dots+d^st_s$,  using the same arguments as in part (1) it can be easily be shown that the set $\{\pi({\hat\rho_{p}}) \ |  \   p\in L_{(d, d^{\infty})}\}\cup \{ \pi(\hat\rho_{d^{\infty}})\}$ spans ${\rm Ext}_{L_K(E)}^1(\Vd,\Vd)$.
\end{proof}

\begin{lemma}\label{Ldqisinfinite}
Let $d$ be a simple closed path in the finite graph $E$.  
\begin{enumerate} 
\item If $q \in E^\infty$ is irrational and  $L_{(d, q)}\neq \emptyset$, then $|L_{(d, q)}|$ is infinite. 
\item If $L_{(d, d^{\infty})}\neq \emptyset$, then $|L_{(d, d^{\infty})}|$ is infinite.
\end{enumerate}
 \end{lemma}
 \begin{proof} 
 (1)
 Let  $q=e_1e_2\cdots$  for $e_i\in E^1$.  First notice that, for any  $w\in E^0$,  if   $w=r(e_i)$ for some $i>0$, then we can assume without loss of generality that $w=r(e_j)$ for infinitely many $j>0$ (as otherwise, since $E^1$ is finite, we can replace $q$ with $q'\in E^\infty$ for which  $q\sim q'$ and $w \not\in (q')^0$).  
 
  Consider now an element  $p\in L_{(d, q)}$. Then $p=\beta q_0$ and $q=\gamma q_0$ for some $q_0 \in [q]$ and $\beta, \gamma \in {\rm Path}(E)$ and $\beta$ not divisible by $d$. Consider  $w=r(\beta)=s(q_0)$. Then, by the previous assumption, there exists a set of infinite and distinct truncations  $\{ \tau_{>n_k}(q) \ |  \ k\in \mathbb{N}\}$   such that, for each $k\in \mathbb{N}$, $q= \gamma_k w \tau_{>n_k}(q)$ for some $\gamma_k \in {\rm Path}(E)$.  Since $q$ is irrational, by Remark \ref{pathsremark} the infinite paths in the set  $\{\tau_{>n_k}(q) \ | \ k\in \N\}$ are distinct. Hence there are infinitely many distinct elements $ \beta \tau_{>n_k}(q)$ in $ L_{(d, q)}$, which establishes (1).
  
  \smallskip
  
 (2) If $L_{(d, d^{\infty})}\neq \emptyset$, then   there is at least one  simple closed path $c$ for which $s(c)=s(d)$ and $c\neq d$.
Then we easily get that each of the distinct paths $\{ c^i d^{\infty} \ | \ i\in \N \}$ is tail equivalent to $d^{\infty}$, which gives that  $L_{(d, d^{\infty})}$ is infinite.
 \end{proof}
 
\begin{corollary}
Let $d$ be a simple closed path in $E$ and $T$ a Chen simple  module. If $\dim_K \Ext^1_{L_K(E)}(\Vd, T)$ is finite, then $T=V_{[c^{\infty}]}$ for a simple closed path $c$.
%
\end{corollary}
\begin{proof}
It follows directly from  Lemma \ref{Ldqisinfinite} and Proposition~\ref{dimExtSTprop}. 
\end{proof}

\begin{example}\label{L(dq)finiteExample} 
{\rm For each $n\in \N$, consider the graph 
  $$E_n \ =  \ \ \xymatrix{ \bullet^v \ar@(ul,dl)_d \ar[r]^{(n)}
& \bullet^w  \ar@(ur,dr)^f} \ \ ,$$
where the symbol $(n)$ indicates that there are $n$ edges $\{ e_1, \dots , e_n\}$ for which  $s(e_i) = v$ and $r(e_i) = w$.   Then in $E_n$ we have $L_{(d,f^\infty)} = \{e_1 f^\infty, \dots , e_nf^\infty\}$, so that $|L_{(d,f^\infty)}| = n$. By Proposition~\ref{dimExtSTprop}(1) we conclude that $\dim_K\Ext_{L_K({E_n})}^1(\Vd,V_{[f^\infty]})=n.$
}
\end{example}

\begin{example}\label{R1Example} 
{\rm Consider the graph 
  $$R_1 \ =  \ \ \xymatrix{ \bullet^v \ar@(ur,dr)^d} \ \ ,$$
for which $L_K(R_1) \cong K[x,x^{-1}]$.  Then $L_{(d,d^\infty)}$ is empty, hence  by Proposition~\ref{dimExtSTprop}(2) we conclude that  $\dim_K\Ext^1_{L_K({R_1})}(\Vd,\Vd)=1.$  
}
\end{example}




%

\smallskip

Having given a complete analysis of ${\rm Ext}_{L_K(E)}^1(\Vd,T)$ for any simple closed path $d$ of $E$ and any Chen simple module $T$, we now analyze ${\rm Ext}_{L_K(E)}^1(\Vp,T)$ for any irrational infinite path $p$ and any  Chen simple module $T$. The projective resolution of $\Vp$ we are going to use is the one introduced in the proof of Theorem~\ref{Vpinftynotfinitelypresented}: 
 $$ \xymatrix{ 0 \ar[r] &  {\rm Ker}(\rho_p) \ar[r] & L_K(E)v \ar[r]^{ \ \ \ \rho_p} &  V_{[p]} \ar[r] & 0},$$
where ${\rm Ker}(\rho_p) =    \oplus_{i=0}^\infty J_i(p)$ as in Corollary \ref{KernelissumofJsubi} and $v=s(p)$. 

\begin{remark}\label{rem:idempotent}
{\rm Let $E$ be a finite graph and $u\in L_K(E)$. For each left $L_K(E)$-module $M$, any morphism $\phi\in\Hom_{L_K(E)}(L_K(E) u, M)$ is the right product by the element $\phi(u)$ of $M$. If $u$ is an idempotent, then it is $\Hom_{L_K(E)}(L_K(E) u, M)\cong uM$ as abelian groups, by means of the isomorphism $\phi\mapsto \phi(u)=u\phi(u)$.}\end{remark}
%
%
%
%
%
%

In order to state the analog of Theorem~\ref{Ext(ST)Srational} for the irrational case we need the following notation. Let $p = e_1e_2\cdots \in E^\infty$ be an irrational infinite path in $E$.  For each $i\geq 0$ let $X_i(p)$ denote the set $\{f \in E^1 \ | \ s(f) = s(e_{i+1}) \ {\rm and}  \ f \neq e_{i+1}\}$  as presented  in Definition \ref{XsubiDefinitions}, and define 
$$r(X_i(p)) := \{w\in E^0 \ | \ w = r(f) \ \mbox{for some} \ f \in X_i(p)\}.$$ 
Finally, for any $p\in E^{\infty}$ and any $i\geq 0$ recall from Definitions \ref{XsubiDefinitions} that the left $L_K(E)$-ideal $J_i(p)$ is 
 $$ J_i(p) = \sum_{ f \in X_i(p)} L_K(E) f^* p_i^* \ ,$$
 where $p_i$ denotes the truncation $\tau_{\leq i}(p)$.
As proved in Lemma~\ref{eachJiisadirectsum}, if $p$ is irrational, then 
$$J_i (p) \ =  \ \bigoplus_{f\in X_i(p)}L_K(E)  f^* p_i^*  \  \cong \ \bigoplus_{f\in X_i(p)}L_K(E)r(f).$$

\smallskip


\begin{lemma}\label{(Li)tequalszero}   Let $p$ be an irrational infinite path in the finite graph $E$.  Let $T$ denote a Chen simple module and let $t \in T$.    Then there exists a positive integer $N = N(t)$ for which $(J_i(p))t = 0$ for all $i\geq N$.  
\end{lemma}

\begin{proof}  Assume $q\in E^{\infty}$ ($q$ can be rational, irrational or a sink) and let $\alpha \in {\rm Path}(E)$ with ${\rm length}(\alpha)=n$. Observe that, for any $i\geq n$,  one has $p_i^*\alpha q=0$ unless  $p_i=\alpha \tau_{\leq {i-n}}(q)$.  So  if  $p_i^*\alpha q\neq 0 $ for all $i\in \mathbb{N}$, we can conclude  $p=\alpha q$. Finally notice that, if there exists $N\in \mathbb{N}$ such that $p_N^*\alpha q= 0 $, then $p_i^*\alpha q= 0 $ for any $i\geq N$.

(Case 1.)  Let   $T\neq V_{[p]}$ and let $q\in E^{\infty}$, necessarily  not tail-equivalent to $p$, such that $T=V_{[q]}$. 
Consider an element $t\in T$. Then $t$ can be written as $t= \sum_{u = 1}^s k_u\alpha_u\tau_{>i_{u}}(q)$, where the  $\alpha_u$'s are  in ${\rm Path}(E)$ and the $i_u$'s are in $\mathbb{N}$, and hence $\alpha_u \tau_{>i_{u}}(q) \neq p$ for any $u=1, \dots, s$.  Since  any element in $(J_i(p))t$ is a finite sum of expressions of the form $k_uf^*p_i^*\alpha_u\tau_{>i_{u}}(q)$, and since  $\alpha_u \tau_{>i_{u}}(q)\neq p$, by the previous observations we can choose an $N=N(t)$ sufficiently large such that $(J_i(p))t = 0$ for any $i\geq N$.

\smallskip

(Case 2.)   On the other hand, let    $T = V_{[p]}$ and  consider an element $t\in T$. Then $t$ can be written as $t= \sum_{u = 1}^s k_u q_u$, where the $q_u$'s  are tail equivalent to $p$. 
Let $f\in X_i(p)$, $i\geq 0$; if $q_u=p$ then $f^*p_i^*p=f^*\tau_{>i}(p)=0$,  by construction of $f^*$. If $q_u\neq p$, there exists an integer $i_u>0$ such that $\tau_{\leq i_u}(q)\neq p_{i_u}$ and hence $p_{i_u}^*\tau_{\leq i_u}(q)=0$.  Then, by the initial observation, if $N=N_t=\max\{i_u: u=1,...,s\}$, we conclude that  $(J_i(p))t = 0$ for any $i\geq N$.
\end{proof}

\smallskip

\begin{lemma}\label{U(T)lemma}  Let $E$ be a finite graph, and let $p$ be an infinite irrational path in $E$.  Let $T$ be a Chen simple $L_K(E)$-module. Then   ${\rm Hom}_{L_K(E)}(J_i(p), T) \neq 0$ if and only if $r(f) \in U(T)$ for some $f\in X_i(p)$.  
\end{lemma}

\begin{proof}  By standard ring theory, we have the following isomorphisms of abelian groups $$\hspace{-2in}{\rm Hom}_{L_K(E)}(J_i(p), T) \cong {\rm Hom}_{L_K(E)}(\oplus_{f\in X_i(p)}{L_K(E)}r(f), T) $$
$$\hspace{.65in}\cong \oplus_{f\in X_i(p)}{\rm Hom}_{L_K(E)}({L_K(E)}r(f), T) \cong \oplus_{f\in X_i(p)}r(f)T,$$

\smallskip
\noindent
where the second isomorphism holds because  $|X_i(p)|$ is finite, and the final one by Remark~\ref{rem:idempotent} because each $r(f)$ is idempotent.
\end{proof}




\begin{theorem}\label{Ext(ST)Sirrational} {\rm (Type (3))}     Let $p$ be an irrational infinite path in the finite graph $E$ and  let $T$ be any Chen simple ${L_K(E)}$-module.  Then ${\rm Ext}^1_{L_K(E)}(\Vp,T) \neq 0 $ if and only if 
 $r(X_i(p)) \cap U(T) \neq \emptyset$ for infinitely many $i\geq 0$.  In such a situation, ${\rm dim}_K({\rm Ext}^1_{L_K(E)}(\Vp,T))$ is infinite.
\end{theorem}

\begin{proof}

($\Rightarrow$) \ 
Suppose $r(X_i(p)) \cap U(T) \neq \emptyset$ for at most finitely many $i\geq 0$.    We seek to show that every element of ${\rm Hom}_{L_K(E)}({\rm Ker}(\rho_p), T)$ arises as right multiplication by an element of $T$.   We have ${\rm Hom}_{L_K(E)}({\rm Ker}(\rho_p), T) = {\rm Hom}_{L_K(E)}(\oplus_{i\geq 0}J_i(p), T)  \cong \prod_{i\geq 0} {\rm Hom}_{L_K(E)}(J_i(p),T)$, which by Lemma \ref{U(T)lemma} and hypothesis equals $ \prod_{i = 0}^N {\rm Hom}_{L_K(E)}(J_i(p),T)$ for some $N \in \N$.   For each $i\geq 0$ and $f\in X_i(p)$, the element $p_i f f^* p_i^*$ is an idempotent generator of $L_K(E) f^* p_i^*$; moreover $\{p_i f f^* p_i^*: f\in X_i(p), i\geq 0\}$ is a set of orthogonal idempotents in $L_K(E)$. Every element 
$\varphi$ of ${\rm Hom}_{L_K(E)}({L_K(E)}p_iff^*p_i^*,T)$ is the right multiplication  by $
\varphi(p_iff^*p_i^*)$; then every element $\psi$ of ${\rm Hom}_{L_K(E)}(\oplus_{i=0}^N J_i(p),T)$ is the right multiplication  by $\psi(\sum_{i=0}^N \sum_{f \in X_i(p)} p_iff^*p_i^*)$.   So ${\rm Ext}^1_{L_K(E)}(\Vp,T) =  0$.
%
%
%
%

\smallskip

($\Leftarrow$)  Conversely, let us see that
$r(X_i(p)) \cap U(T) \neq \emptyset$ for infinitely many $i\geq 0$ implies that there is an element of ${\rm Hom}_{L_K(E)}({\rm Ker}(\rho_p), T)$ which does not arise as a right multiplication by an element of $T$.  By Lemma~\ref{U(T)lemma} there exists an increasing sequence $(i_n)_{n\in\mathbb N}$ of natural numbers such that ${\rm Hom}_{L_K(E)}(J_{i}(p), T) \neq 0$ if and only if $i=i_n$ for a suitable $n\in\mathbb N$. Let $\{\phi_i\in {\rm Hom}_{L_K(E)}(J_{i}(p), T):i\in\mathbb N\}$ be a family of morphisms such that $\phi_{i_n}\not=0$ for each $n\in\mathbb N$.
Then \[\varphi=\prod_{i\in\mathbb N}\phi_i\in{\rm Hom}_{L_K(E)}(\oplus_{i\in\mathbb N}J_{i}(p), T)\]
is a morphism which is not, by Lemma~\ref{(Li)tequalszero}, right multiplication by  element of $T$.
%
%
%
%
%

\smallskip

To establish the final statement, consider an increasing sequence $(i_n)_{n\in\mathbb N}$ of natural numbers and a family $\{\phi_{i_n}\in {\rm Hom}_{L_K(E)}(J_{i_n}(p), T):n\in\mathbb N\}$ of nonzero morphisms. Define for each prime $z\in\mathbb N$ the morphism  $\Psi^{z}\in \Hom_{L_K(E)}(\Ker(\rho_p), T)$ as follows:  for each $j\geq 0$,  
$$\Psi^{z}( J_{i_j}(p))=\psi_{i_j}(J_{i_j}(p)) \mbox { if } z \mbox{ divides }  j, \  \mbox { while } \Psi^{z}( J_{\ell}(p))=0 \mbox{ otherwise. }$$
Observe that $\pi(\Psi^{z})\neq 0$, since $\Psi^{z}$ has infinitely many nonzero components.  
Finally, $\{\pi(\Psi^{z}) \ | \ z\in\mathbb N, \ z\text{ prime} \}$ is a set of linearly independent elements of 
$\Ext^1_{L_K(E)}(\Vp, T))$, as follows.  Let $F$ be a finite subset of primes in $\mathbb N$, and assume $\sum_{z\in F}k_z\pi(\Psi^{z})=0$; then $\pi(\sum_{z\in F}k_z \Psi^{z})=0$ and therefore $\sum_{z\in F}k_z \Psi^{z}$ is a right multiplication by an element $t$ of $T$. By Lemma~\ref{(Li)tequalszero} there exists a positive integer $N = N(t)$ for which $(J_i(p))t = 0$ for all $i\geq N$. For each prime $\hat z\in F$, let $m_{\hat z}$ a natural number greater than $N$; then
\[0=(J_{{\hat z}^{m_{\hat z}}}(p))t=\sum_{z\in F}k_z \Psi^{z}(J_{{\hat z}^{m_{\hat z}}}(p))=
k_{\hat z}\Psi^{\hat z}(J_{{\hat z}^{m_{\hat z}}}(p))\]
and hence $k_{\hat z}=0$. Hence $\dim_K(\Ext_{L_K(E)}^1(\Vp, T))$ is infinite. 
\end{proof}

\bigskip

\noindent We emphasize the fact that Theorems~\ref{Ext(ST)Srational} and \ref{Ext(ST)Sirrational}    allow us to compute the dimension of the Ext$^1$-groups between two Chen simple modules completely and solely in terms of  properties of the graph $E$.

\begin{example}\label{example:infinitedimensionalExt}
{\rm We again revisit the graph $R_2$ and irrational infinite path $q=efeffefffe \cdots$ of Example \ref{R2Example}.  Let $T = V_{[q]}$.  Since clearly $U(T) = \{v\}$ and $r(X_i) = \{v\} $ for all $i\in \N$ as well, Theorem \ref{Ext(ST)Sirrational} yields that $\dim_K(\Ext_{L_K(E)}^1(V_{[q]},V_{[q]} ))$ is infinite. 
}
\end{example}

\begin{example}\label{example:irrationalfinitedimensional}
{\rm 
With the statement of Theorem \ref{Ext(ST)Sirrational} as motivation, we give examples of graphs $E_n$, an irrational infinite path $p$ in $E_n$, and Chen simple $L_K(E_n)$-modules $T$ having $r(X_i(p)) \cap U(T)  \neq \emptyset$ for only finitely many $i\in \Z^+$.   For $n \in \N$ consider the graph $E_n$ given by
$$  
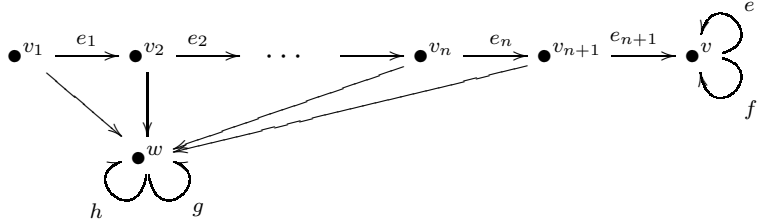
\xymatrix{  \bullet^{v_1} \ar[r]^{e_1}  \ar[dr]  & \bullet^{v_2}  \ar[r]^{e_2}  \ar[d]   &   & \hspace{-.5in}  \cdots  \ar[r] & \bullet^{v_n} \ar[r]^{e_n}  \ar[dlll] & \bullet^{v_{n+1}} \ar[r]^{e_{n+1}}  \ar[dllll]& \bullet^{v} \ar@(r,u)_e \ar@(r,d)^f \\
 & \bullet^w \ar@(d,r)_g \ar@(d,l)^h  &  &   &   &   &    }$$
Let $p$ denote the  irrational infinite path $e_1e_2 \cdots e_n e_{n+1} efeffefffe \cdots$.   Let $T_1$ be the Chen simple $L_K(E)$-module $V_{[g^\infty]}$,
and let $T_2$ denote the Chen simple $L_K(E)$-module $V_{[q]}$ corresponding to the  irrational infinite path $q = ghghhghhhg \cdots$.   Then for $j=1,2$, $r(X_i(p)) \cap U(T_j) $ is nonempty (indeed, equals $\{w\}$) precisely when $0\leq i \leq n$.   

Consequently,  by Theorem \ref{Ext(ST)Sirrational}, $\Ext_{L_K(E)}^1(V_{[p]},T_j) = \{0\}$ for $j=1,2$. 
}
\end{example}

We conclude the article by demonstrating the existence of indecomposable $L_K(E)$-modules of prescribed finite length, in case $E$ is a finite graph which contains cycles.  
Recall that a module $M$ is called   {\it uniserial} in case the lattice of submodules of $M$ is totally ordered.  In particular, any uniserial module is indecomposable. Moreover, the radical $\Rad(M)$  of a uniserial module $M$ is the unique maximal submodule of $M$, hence $M/\Rad{M}$ is simple.

\begin{lemma}\cite[Lemma 16.1 with Proposition 16.2]{LR}\label{lemma:uniserial}
Let $R$ be any unital ring.  Let  $U$ be a uniserial left $R$-module of finite length, and $X$ a simple left $R$-module.
Consider the morphism $\psi:\Ext^1_R(X,U)\to \Ext^1_R(X,U/\Rad U)$. An extension in $\Ext^1_R(X,U)$ is uniserial if and only if it does not belong to $\Ker\psi$.

In particular, if $R$ is hereditary, there exists a uniserial extension of $U$ by $X$ if and only if $\Ext^1_R(X,U/\Rad U)\not=0$.
\end{lemma}
%

As observed in Remark \ref{rem:Bergman}, $L_K(E)$ is hereditary for any row-finite graph $E$.  So Lemma \ref{lemma:uniserial} gives the following:

\begin{corollary}\label{cor:uniserialoflengthN}
Let $E$ be a finite graph. If $S$ is a Chen simple $L_K(E)$-module such that $\Ext^1_{L_K(E)}(S, S)\neq 0$ and $L$ is a uniserial $L_K(E)$-module such that $L/\Rad(L)\cong S$, then there exists a uniserial $L_K(E)$-module $M$ which is an extension of $L$ by $S$.

In particular, for any $n\in \mathbb{N}$ there exists a uniserial  $L_K(E)$-module of length  $n$, all of whose  composition factors are isomorphic to $S$.
%
%
%
%
\end{corollary}
\begin{proof}
The first statement follows directly from Lemma~\ref{lemma:uniserial} and the hereditariness of $L_K(E)$. In order to show the existence of uniserial modules of arbitrary length, first observe that any non-zero element of the abelian group $\Ext^1_{L_K(E)}(S, S)\neq 0$ corresponds to an indecomposable uniserial module $L_2$ of length $2$, with $\Rad(L_2)\cong S$ and $L_2/{\Rad(L_2)}\cong S$. Then, by applying the first statement, there exists a uniserial module $L_3$ of length $3$ which is an extension of $L_2$ by $S$. Since $\Rad(L_3)\cong L_2$ and hence $L_3/\Rad(L_3)\cong S$, we can proceed by induction.
\end{proof}

Observe that if $E$ contains a simple closed path $d$, by Corollary \ref{cor:Ext1VdVd} the module $S=V_{[d^{\infty}]}$ satisfies $\Ext^1_{L_K(E)}(S, S)\neq 0$ and hence Corollary~\ref{cor:uniserialoflengthN} applies.


\end{document}